\documentclass[aihp]{imsart}

\RequirePackage{amsthm,amsmath,amsfonts,amssymb}
\RequirePackage[numbers]{natbib}
\RequirePackage[colorlinks,citecolor=blue,urlcolor=blue]{hyperref}
\RequirePackage{graphicx}

\startlocaldefs
\theoremstyle{plain}

\newtheorem{theorem}{Theorem}[section]
\newtheorem{lemma}[theorem]{Lemma}
\newtheorem{assumption}[theorem]{Assumption}
\newtheorem{proposition}[theorem]{Proposition}
\newtheorem{corollary}[theorem]{Corollary}
\theoremstyle{remark}
\newtheorem{remark}[theorem]{Remark}

\theoremstyle{remark}
\newtheorem{definition}[theorem]{Definition}

\usepackage{amssymb, amsmath,mathtools}
\usepackage{mathrsfs}
\usepackage{amscd}
\usepackage{verbatim}
\usepackage{stmaryrd}
\usepackage[dvipsnames]{xcolor}

\usepackage{enumerate}

\usepackage{forest}
\usepackage{tikz}
\tikzset{>=latex}

\usepackage{tabularx}
\newcolumntype{L}{>{\arraybackslash}X}
\usepackage{multirow}

\numberwithin{equation}{section}

\def\N{{\mathbb N}}

\def\R{{\mathbb R}}

\renewcommand{\P}{{\mathbb P}}
\newcommand{\F}{{\mathscr F}}

\newcommand{\g}{\gamma}

\renewcommand{\O}{\Omega}

\renewcommand{\a}{\kappa}

\newcommand{\Dom}{\mathscr{O}}

\newcommand{\I}{I}

\newcommand{\Tor}{\mathbb{T}}
\newcommand{\T}{\mathbb{T}}

\newcommand{\A}{{\mathcal A}}
\renewcommand{\b}{{\mathcal B}}

\newcommand{\calL}{{\mathscr L}}
\newcommand{\Sc}{{\mathcal S}}

\newcommand{\hz}{\prescript{}{0}{H}}

\newcommand{\Wz}{\prescript{}{0}{W}}

\newcommand{\Sol}{\mathscr{R}}

\DeclareMathOperator*{\esssup}{\textup{ess\,sup}}

\renewcommand{\AA}{\mathcal{T}}

\newcommand{\Do}{\mathrm{D}}

\newcommand{\tT}{T^{*}}

\newcommand{\wt}{\widetilde}

\usepackage{stmaryrd}

\newcommand{\MRtas}{\mathcal{SMR}_{p,\a}^{\bullet}(s,T)}
\newcommand{\MRtazero}{\mathcal{SMR}_{p,0}^{\bullet}(s,T)}

\newcommand{\MRtt}{\mathcal{SMR}_{p}^{\bullet}(t,T)}

\newcommand{\MRta}{\mathcal{SMR}_{p,\a}^{\bullet}(T)}

\newcommand{\MRtasz}{\mathcal{SMR}_{p,\a}(s,T)}

\newcommand{\MRttwoszero}{\mathcal{SMR}_{2,0}^{\bullet}(s,T)}

\newcommand{\MRttwosz}{\mathcal{SMR}_{2,0}(s,T)}

\newcommand{\MRtasonesequencezer}{\mathcal{SMR}_{p,\a}(s_{1})}
\newcommand{\MRtasjsequencez}{\mathcal{SMR}_{p}(s_j,s_{j+1})}
\newcommand{\MRtasonesequencez}{\mathcal{SMR}_{p,\a}(s,s_{1})}
\newcommand{\Tr}{\mathrm{Tr}}

\newcommand{\Xap}{X^{\mathrm{Tr}}_{\a,p}}

\newcommand{\Xp}{X^{\mathrm{Tr}}_{p}}

\newcommand{\Xzp}{X^{\mathrm{Tr}}_{0,p}}

\newcommand{\one}{{{\bf 1}}}
\newcommand{\embed}{\hookrightarrow}

\newcommand{\B}{B}

\newcommand{\supp}{\mathrm{supp}\,}

\renewcommand{\l}{\langle}
\renewcommand{\r}{\rangle}
\newcommand{\deter}{\mathrm{det}}

\newcommand{\stoc}{\mathrm{sto}}
\newcommand{\Fou}{\mathcal{F}}
\newcommand{\reg}{\delta}

\newcommand{\norm}[1]{{\left\vert\kern-0.25ex\left\vert\kern-0.25ex\left\vert #1
    \right\vert\kern-0.25ex\right\vert\kern-0.25ex\right\vert}}

\newcommand{\Progress}{\mathscr{P}}

\newcommand{\wh}{\widehat}

\newcommand{\btwod}{b}

\newcommand{\am}{a}
\newcommand{\bm}{b}

\newcommand{\modulus}{\Theta}
\newcommand{\Borel}{\mathscr{B}}

\newcommand{\e}{\mathrm{E}}

\def\XXint#1#2#3{{\setbox0=\hbox{$#1{#2#3}{\int}$ }
\vcenter{\hbox{$#2#3$ }}\kern-.6\wd0}}

\newcommand{\Set}{\mathcal{S}}

\newcommand{\dd}{\mathrm{d}}

\allowdisplaybreaks

\endlocaldefs

\begin{document}

\begin{frontmatter}
\title{Stochastic maximal $L^p(L^q)$-regularity for second order systems with periodic boundary conditions}
\runtitle{Stochastic maximal $L^p(L^q)$-regularity}

\begin{aug}
\author[A,C]{\inits{A.}\fnms{Antonio}~\snm{Agresti}\ead[label=e1]{antonio.agresti92@gmail.com}\orcid{0000-0002-9573-2962}}
\and
\author[B]{\inits{M.C.}\fnms{Mark}~\snm{Veraar}\ead[label=e2]{m.c.veraar@tudelft.nl}\orcid{0000-0003-3167-7471}}
\address[A]{
Department of Mathematics,
	Technische Universit\"{a}t Kaiserslautern, Kaiserslautern, Germany
}

\address[C]{
current address: Institute of Science and Technology Austria (ISTA), Klosterneuburg, Austria
\printead[presep={,\ }]{e1}}

\address[B]{Delft Institute of Applied Mathematics, Delft University of Technology, Delft, The
Netherlands\printead[presep={,\ }]{e2}}
\end{aug}

\begin{abstract}
In this paper we consider an SPDE where the leading term is a second order operator with periodic boundary conditions, coefficients which are measurable in $(t,\omega)$, and H\"older continuous in space. Assuming stochastic parabolicity conditions, we prove $L^p((0,T)\times \Omega, t^{\kappa} \mathrm{d} t;H^{\sigma,q}(\T^d))$-estimates. The main novelty is that we do not require $p=q$. Moreover, we allow arbitrary $\sigma\in \R$ and weights in time. Such mixed regularity estimates play a crucial role in applications to nonlinear SPDEs which is clear from our previous work. To prove our main results we develop a general perturbation theory for SPDEs. Moreover, we prove a new result on pointwise multiplication in spaces with fractional  smoothness.
\end{abstract}

\begin{abstract}[language=french]
Dans cet article, nous considérons une EDPS où le terme dominant est un opérateur du second ordre avec des conditions aux limites périodiques, des coefficients mesurables en $(t,\omega)$ et de regularité H\"olderienne en espace. En supposant des conditions de parabolicité stochastique, nous prouvons des estimations de type $ L^p((0,T)\times \Omega, t^{\kappa}  \mathrm{d} t;H^{\sigma,q}(\T^d))$. La principale nouveauté est que nous n'avons pas besoin de $p=q $. De plus, nous autorisons $\sigma\in \R$ à \^etre arbitraire et des poids en temps. De telles estimations de régularité mixtes jouent un rôle crucial dans les applications aux EDPS non linéaires, ce qui ressort clairement de nos travaux précédents. Pour prouver nos principaux résultats, nous développons une méthode générale perturbative pour les EDPS. De plus, nous prouvons un nouveau résultat sur la multiplication ponctuelle dans des espaces à régularité fractionnaire.
\end{abstract}



\end{frontmatter}

\section{Introduction}
In this paper we consider second order systems on the torus $\T^d$:
\begin{equation}
\label{eq:parabolic_problem_intro}
\begin{cases}
\dd u +\A \, u\, \dd t= f\, \dd t + \sum_{n\geq 1} (\b_n u+g_n)\, \dd w^n_t,&\text{ on }\Tor^d,\\
u(0)=u_0, &\text{ on }\Tor^d.
\end{cases}
\end{equation}
Here $\A$ is a second order operator, $\b_n$ are first order operators, and we suppose  that a stochastic parabolicity condition holds (see e.g.\ \eqref{eq:stochparintro} below). The processes $(w^n)_{n\geq 1}$ are independent standard Brownian motions on a filtered probability space.
Our goal is to prove optimal regularity estimates for the solution to \eqref{eq:parabolic_problem_intro} in a weighted $L^p$-setting in time and in $H^{\sigma,q}$ in space, where $\sigma\in \R$ and $p,q\in [2, \infty)$ are arbitrary. Our motivation for considering optimal regularity estimates for linear problems such as \eqref{eq:parabolic_problem_intro} comes from the applications to nonlinear SPDEs which were recently obtained in \cite{AV19_QSEE_1,AV19_QSEE_2}.
In particular, from these works it is clear that having full flexibility in $\sigma,p,q$ is important when considering critical spaces for nonlinear SPDEs (also see \cite{CriticalQuasilinear} for the deterministic case).

Unfortunately, in the stochastic case there are very few results available in the case $p\neq q$. The aim of this paper is to obtain sufficient conditions for optimal regularity results with as much flexibility as possible.
In the case the coefficients of $\A$ and $\b$ are only dependent on $x$ (space), and $b = 0$ (or small), $L^p(L^q)$-theory can be deduced from \cite{MaximalLpregularity}, where a sufficient conditions for stochastic maximal regularity is given in terms of the $H^\infty$-functional calculus of the leading differential operator. Moreover, in \cite{LoVer} an extrapolation result was obtained which can be used to go from an $L^q(L^q)$-setting to an $L^p(L^q)$-setting.

Unfortunately, both of the latter results cannot directly be used if the coefficients are $(t,\omega,x)$-dependent and $\b$ is not small.
However, if $p=q$, then many sufficient conditions are available in the literature (see \cite{KimLeesystems,Kry,VP18} and references therein). Moreover, in \cite{VP18} the case $p\neq q$ is covered if the coefficients are only $(t,\omega)$-dependent.
The difficulties in case $p\neq q$ are due to the fact that Fubini arguments fail, and this has consequences for localization arguments. In \cite{GV17} this problem was circumvented for deterministic PDEs with continuous coefficient by using the weighted extrapolation theory of Rubio de Francia. This was further improved to VMO-coefficients in space in \cite{DK18}. At the moment the latter approach seems not available for SPDEs.

In the current paper we give suitable conditions under which a localization argument can be done in the $p\neq q$ setting. Here we only assume regularity conditions on the coefficients in the space variable. In $(t,\omega)$ the coefficients are only assumed to be progressively measurable. Our starting point is the case of space-independent coefficients previously considered in \cite[Theorem 5.3]{VP18}. In the case of $(t,\omega,x)$-dependent coefficient we obtain stochastic maximal $L^p$-regularity using a new abstract perturbation result for stochastic maximal $L^p$-regularity, and new results on pointwise multiplication in spaces of fractional smoothness, we are able to apply the general framework of \cite{VP18} combined with some results from \cite{AV19_QSEE_1,AV19_QSEE_2}, to do a localization argument to build in $x$-dependence.

In the preprint \cite{AV20_NS} we extend the method of the current paper to obtain a similar result for the stochastic Stokes equations. Using \cite{AV19_QSEE_1} we obtain local well-posedness (local for $d\geq 3$ and global for $d=2$) for the stochastic Navier-Stokes equations with a transport noise term in an $L^p(L^q)$-setting. Via \cite{AV19_QSEE_2} we then prove new regularity properties of the solution. The results of the current paper will be used in \cite{AV_reaction} to obtain new well-posedness and regularity results for reaction-diffusion equations. Here the $L^p(L^q)$-setting will play a crucial role.

To conclude this introduction we state a special case of our main result in the case of scalar equations (Theorem \ref{t:parabolic_problems} for $m=1$). Below $\Progress$ and $\Borel$ denote the progressive and the Borel $\sigma$-algebra, respectively.
\begin{assumption}\label{ass:intro}
Let $\sigma\in\R$ and $T\in (0,\infty)$. 
For each $i,j\in \{1,\dots,d\}$ and $n\geq 1$, $\am^{i,j}:[0,T]\times \O\times \Tor^d\to \R$ and $\bm^{j}_{n}:[0,T]\times \O\times \Tor^d\to \R$ are $\Progress\otimes \Borel(\Tor^d)$-measurable, and there exist $C_{\am,\bm}>0$, $\alpha>|1+\sigma|$ such that a.s.\
\begin{align*}
\|\am^{i,j}(t,\cdot)\|_{C^{\alpha}(\Tor^d)}+\|(\bm^{j}_{n}(t,\cdot))_{n\geq 1}\|_{C^{\alpha}(\Tor^d;\ell^2)}&\leq C_{\am,\bm}, \  t\in [0,T],  \ i,j\in \{1,\dots,d\}.
\end{align*}
There exists $\vartheta>0$ such that, a.s.\ for all $t\in [0,T]$, $x\in \Tor^d$, $\eta\in \R^m$ and $\xi\in \R^d$,
\begin{align}\label{eq:stochparintro}
\sum_{i,j=1}^d \Big(\am^{i,j}(t,x)-\frac{1}{2}\sum_{n\geq 1} \bm^j_{n}(t,x)\bm^i_{n}(t,x)\Big)  \xi_i \xi_j\geq \vartheta |\xi|^2.
\end{align}
For $n\geq 1$ and $t\in [0,T]$, set
\begin{equation}
\label{eq:def_AB_parabolicintro}
\begin{aligned}
\A(t) u:=-\sum_{i,j=1}^d \partial_i (\am^{i,j}(t,\cdot)\partial_ju) 
\qquad \text{and} \qquad
\b_n(t) u:=\sum_{j=1}^d \bm^{j}_{n}(t,\cdot)\partial_ju
\end{aligned}
\end{equation}
\end{assumption}

The following follows from Theorem \ref{t:parabolic_problems}, where also systems and mixed space-time regularity are considered.

\begin{theorem}
\label{t:parabolic_problems-intro}
Suppose that Assumption \ref{ass:intro} holds.
Let $p\in (2, \infty)$, $q\in [2, \infty)$ and $\kappa\in [0,p/2-1)$ (or $p=q=2$ and $\kappa=0$). Let $w_{\a}(t) = t^{\a}$. Then for each
\begin{align*}
u_0 & \in L^p_{\F_0}(\Omega;B^{2+\sigma-2(1+\kappa)/p}_{q,p}(\mathbb{T}^d))
\\ f & \in  L^p_{\Progress}((0,T)\times \Omega,w_{\a};H^{\sigma,q}(\Tor^d)) \ \ \text{and} \ \
g\in L^p_{\Progress}((0,T)\times \Omega,w_{\a};H^{1+\sigma,q}(\Tor^d;\ell^2)),
\end{align*}
there exists a unique strong solution $u$ to \eqref{eq:parabolic_problem_intro} and
\begin{align*}
\|u\|_{L^p((0,T)\times \Omega,w_{\a};H^{2+\sigma,q}(\Tor^d)))} &\lesssim
\|u_0\|_{L^p(\Omega;B^{2+\sigma-2(1+\kappa)/p}_{q,p}(\mathbb{T}^d))} + \|f\|_{L^p((0,T)\times \Omega,w_{\a};H^{\sigma,q}(\Tor^d))} \\ & \qquad +\|g\|_{L^p((0,T)\times \Omega,w_{\a};H^{1+\sigma,q}(\Tor^d;\ell^2))}
\end{align*}
with implicit constant independent of $(u_0,f,g)$.
\end{theorem}
A similar result holds for $\R^d$ in case the coefficients become constant for $|x|\to \infty$ (see Remark \ref{rem:Rdcase}). Furthermore, a version of the result also holds for non-divergence type operators (see Theorem \ref{t:parabolic_problems_non}).

\section{Preliminaries}\label{s:preliminaries}

In this section we collect known facts and fix our notation. Here $X$ denotes a Banach space.

\subsection{Function spaces}
For details on vector-valued functions spaces and complex interpolation we refer to \cite{AmannII,Analysis1,MV12,Tr1}.
For an open subset $\Dom\subseteq \R^d$ and $k\in \N_0$, $C^k(\overline{\Dom};X)$ denotes the space of all $k$-times continuous differentiable functions with bounded derivatives endowed with its supremum norm. For $\alpha = k+s$ with $k\in \N_0$ and $s\in (0,1)$ define $C^{\alpha}(\overline{\Dom};X)$ to be the space of functions $f\in C^{k}(\overline{\Dom};X)$ for which the derivatives of order $\leq k$ are $s$-H\"older continuous.
For $p,q\in (1,\infty)$ and $\sigma\in \R$, we denote by $H^{\sigma,q}(\Tor^d)$ and $B^{\sigma}_{q,p}(\Tor^d)$, the Bessel potential and Besov spaces, respectively. Moreover, we set $H^{\sigma,q}(\Tor^d;\R^m):=(H^{\sigma,q}(\Tor^d))^m$ and $B^{\sigma}_{q,p}(\Tor^d;\R^{m}):=(B^{\sigma}_{q,p}(\Tor^d))^m$ for all integers $m\geq 1$.

For $p\in (1,\infty)$, $\a\in (-1,p-1)$ and $a\geq 0$, $w_{\a}^a$ denotes the shifted power weight
\begin{equation*}
w_{\a}^a(t):=|t-a|^{\a}, \quad t\in \R,\quad\qquad w_{\a}:=w_{\a}^0.
\end{equation*}

For $\I=(a,b)$ where $0\leq a<b\leq \infty$, let
$L^p(\I,w_{\a}^a;X)$ be the space of all strongly measurable functions $f:\I\to X$ such that
\[\|f\|_{L^p(\I,w_{\a}^a;X)}:=\Big(\int_a^b \|f(t)\|_{X}^p w_{\a}^a(t)\,\dd t\Big)^{1/p}<\infty,\]
where we omit the weight from the above notation if $\kappa=0$. Let
$W^{1,p}(\I,w_{\a}^a;X)$ denote the subspace of $L^p(\I,w_{\a}^a;X)$ such that the weak derivative satisfies $f'\in L^p(\I,w_{\a}^a;X)$. This space is endowed with the norm:
$$
\|f\|_{W^{1,p}(\I,w_{\a}^a;X)}:=\|f\|_{L^{p}(\I,w_{\a}^a;X)}+\|f'\|_{L^{p}(\I,w_{\a}^a;X)}.
$$
Furthermore, for $\theta\in (0,1)$ let
\begin{align*}
\Wz^{1,p}(\I,w_{\a}^a;X)&= \{f\in W^{1,p}(\I,w_{\a}^a;X): f(a)=0\},
\\ H^{\theta,p}(\I,w_{\a}^a;X) & = [L^p(\I,w_{\a}^a;X),W^{1,p}(\I,w_{\a}^a;X)]_{\theta} \ \ \text{(complex interpolation)},
\\ \hz^{\theta,p}(\I,w_{\a}^a;X) &= [L^p(\I,w_{\a}^a;X_1),\Wz^{1,p}(\I,w_{\a}^a;X)]_{\theta}.
\end{align*}
Since each $f\in W^{1,p}(\I,w_{\a}^a;X)$ has a continuous version on $\overline{I}$ (see \cite[Lemma 3.1]{LV18}), the value $f(a)$ is well-defined.

For the definition of complex interpolation we refer to e.g.\ \cite{BeLo,Tr1}. Below we also employ the real interpolation functor which will be denoted by $(\cdot,\cdot)_{\theta,p}$ for $\theta\in (0,1)$ and $p\in (1,\infty)$.

\subsection{Stochastic maximal $L^p$-regularity}

Next we will introduce the main abstract setting to define stochastic maximal $L^p$-regularity.
\begin{assumption}
\label{ass:X}
Let $X_0,X_1$ be UMD Banach spaces with type 2 and $X_1\hookrightarrow X_0$ densely. Assume that one of the following two settings is satisfied
\begin{itemize}
\item $p\in (2,\infty)$ and $\a\in [0,\frac{p}{2}-1)$;
\item $p=2$, $\a=0$ and $X_0,X_1$ are Hilbert spaces.
\end{itemize}
For $\theta\in (0,1)$ and $p,\a$ as above, we set
$$
X_{\theta}:=[X_0,X_1]_{\theta},\qquad \Xap:=(X_0,X_1)_{1-\frac{1+\a}{p},p},\qquad  \Xp:=\Xzp.$$
\end{assumption}

\begin{assumption}
\label{ass:AB_boundedness}
Let $T\in (0,\infty]$ and $s\in [0,T)$. Assume that $A:[s,T]\times\Omega\to \calL(X_1,X_0)$ and $B:[s,T]\times\Omega \to \calL(X_1,\g(H,X_{1/2}))$ are strongly progressively measurable and there exists $C_{A,B}>0$ such that, a.s.\ for all $t\in [s,T]$,
$$
\|A(t,\cdot)\|_{\calL(X_1,X_0)} + \|B(t,\cdot)\|_{ \calL(X_1,\g(H,X_{1/2}))}
\leq
C_{A,B}<\infty.
$$
\end{assumption}
Stochastic maximal $L^p$-regularity is concerned with the optimal regularity estimate for the linear abstract stochastic Cauchy problem:
\begin{equation}
\label{eq:diffAB_s}
\begin{cases}
\dd u(t) +A(t)u(t)\, \dd t=f(t) \, \dd t+ (B(t)u(t)+g(t))\, \dd W_H(t),\quad t\in [s,T],\\
u(s)=u_{s}.
\end{cases}
\end{equation}
Here $W_H$ denotes a $H$-cylindrical Brownian motion (see e.g.\ \cite[Definition 2.11]{AV19_QSEE_1}) on a filtered probability space $(\O,\mathscr{A},(\F_t)_{t\geq 0},\P)$ and $H$ is a separable Hilbert space.
In order to define strong solutions to \eqref{eq:diffAB_s} let $\tau$ be a stopping time such that $s\leq \tau\leq T$ a.s.\ and let
$u_{s}\in L^0_{\F_{s}}(\O;X_0)$, $f\in L^0_{\Progress}(\O;L^1(s,\tau;X_0))$, $g\in L^0_{\Progress}(\O;L^2(s,\tau;\g(H,X_{0})))$, where $\Progress$ denotes the progressive $\sigma$-algebra. A strongly progressive measurable map $u:[s,\tau]\times\Omega\to X_1$ is said to be a \textit{strong solution} to \eqref{eq:diffAB_s} (on $[s,\tau]$) if $u\in L^0(\O;L^2(s,\tau;X_1))$ and, a.s.\ for all $t\in [s,\tau]$,
\begin{equation*}
u(t)-u_{s}+\int_{s}^t A(r)u(r) \, \dd r=\int_{s}^t
 f(r) \, \dd r+ \int_s^t\one_{(s,\tau)}(B(r)u(r)+g(r))\, \dd W_H(r).
\end{equation*}
In the above $\gamma(H,X_{0})$ denotes the set of $\gamma$-radonifying operators from $H$ to $X_0$ (see e.g.\ \cite[Chapter 9]{Analysis2}), and $[0,\tau]\times \Omega:=\{(t,\omega)\in [0,T]\times \Omega\,:\,0\leq \tau(\omega)\leq T\}$. Below we employ a similar notation for $[0,\tau)\times \Omega$, $(0,\tau)\times \Omega$ etc.

As in \cite{AV19_QSEE_1}, we define (weighted) stochastic maximal $L^p$-regularity.

\begin{definition}[Stochastic maximal $L^p$-regularity]
\label{def:SMRgeneralized}
Let Assumptions \ref{ass:X} and \ref{ass:AB_boundedness} be satisfied. We write $(A,B)\in \MRtasz$ if for every
\begin{equation*}
f\in  L^p_{\Progress}( (s,T)\times\Omega,w_{\a}^{s};X_0),
\quad\text{ and }\quad
g\in L^p_{\Progress}((s,T)\times\Omega,w_{\a}^s;\g(H,X_{1/2}))
\end{equation*}
there exists a strong solution $u$ to \eqref{eq:diffAB_s} with $u_s=0$ such that $u\in L^p_{\Progress}((s,T)\times\Omega, w_{\a};X_{1})$, and moreover for all stopping time $\tau$, such that $s\leq \tau \leq T$ a.s., all strong solutions $u\in L^p_{\Progress}((s,\tau)\times \O,w_{\a}^s;X_1)$ to \eqref{eq:diffAB_s} with $u_s=0$ satisfy the estimate
\begin{equation*}
\begin{aligned}
\|u\|_{L^p((s,\tau)\times\Omega,w_{\a};X_1)}
&\lesssim\|f\|_{L^p((s,\tau)\times\Omega,w_{\a}^{s};X_0)}
+\|g\|_{L^p((s,\tau)\times\Omega,w_{\a}^{s};\g(H,X_{1/2}))},
\end{aligned}
\end{equation*}
where the implicit constant is independent of $(f,g,\tau)$. In addition:
\begin{enumerate}[{\rm(1)}]
\item\label{it:SMR_regularity_0} If $p\in (2,\infty)$ and $\a\in [0,\frac{p}{2}-1)$, then we say $(A,B)\in \MRtas$ if $(A,B)\in \MRtasz$ and, for each $\theta\in [0,\frac{1}{2})\setminus\{\frac{1+\a}{p}\}$,
\begin{equation*}
\begin{aligned}
\|u\|_{L^p(\O;\hz^{\theta,p}(s,T,w_{\a}^{s};X_{1-\theta}))}
&\lesssim\|f\|_{L^p((s,T)\times\Omega,w_{\a}^{s};X_0)}+\|g\|_{L^p((s,T)\times\Omega,w_{\a}^{s};\g(H,X_{1/2}))},
\end{aligned}
\end{equation*}
where the implicit constant is independent of $(f,g,\tau)$.
\item If $p=2$ and $\a=0$, then we say $(A,B)\in \MRttwoszero$ if $(A,B)\in \MRttwosz$ and there exists $C>0$ such that
\begin{align*}
\|u\|_{ L^2(\O;C([s,T];X_{1/2}))}
\lesssim \|f\|_{L^2((s,T)\times\Omega ;X_0)}+\|g\|_{L^2((s,T)\times\Omega;\g(H,X_{1/2}))},
\end{align*}
where the implicit constant is independent of $(f,g,\tau)$.
\end{enumerate}
We write $\mathcal{SMR}_p^{\bullet}(s,T):=\MRtazero$ and $A\in \MRtas$ if $(A,0)\in \MRtas$ and similarly if ``$\bullet$'' is omitted.
\end{definition}
If one of the stochastic maximal $L^p$-regularity estimates of Definition \ref{def:SMRgeneralized} holds, then one can also consider \eqref{eq:diffAB_s} with nonzero initial data $u_s\in L^p_{\F_0}(\Omega;\Xap)$ (as defined in Assumption \ref{ass:X}). The above estimates then hold if $\|u_s\|_{L^p(\Omega;\Xap)}$ is added on the right-hand side (see \cite[Proposition 3.10]{AV19_QSEE_1}).

Examples of operators with stochastic maximal regularity appear in many places in the literature, and we refer to the examples in \cite[Subsection 3.2]{AV19_QSEE_1} and references therein. In Section \ref{sec:SMRsecond} we present a general sufficient conditions for stochastic maximal regularity for second order operators on $\T^d$ and $\R^d$.

Following \cite{AV19_QSEE_1,AV19_QSEE_2}, we introduce the \textit{solution operator} $\Sol_{s,(A,B)}$ associated to the couple $(A,B)\in \MRtas$ defined as
\begin{equation}
\label{eq:solution_operator_definition}
u:=\Sol_{s,(A,B)}(0,f,g)
\end{equation}
where $u$ is the unique solution to \eqref{eq:diffAB_s}. Then $\Sol_{s,(A,B)}(0,\cdot,\cdot)$ defines a mapping
$$
L^p((s,T)\times\Omega,w_{\a}^s;X_0)\times
L^p((s,T)\times\Omega,w_{\a}^s;\g(H,X_{1/2}))\to
L^p(\O;H^{\theta,p}(s,T,w_{\a}^s;X_{1-\theta}))
$$
provided $p>2$ and $\theta\in [0,\frac{p}{2}-1)$. Here and below, if $p=2$ and $\theta\in (0,\frac{1}{2})$, then one has to replace
$H^{\theta,p}(s,T,w_{\a}^s;X_{1-\theta})$ by $C([s,T];X_{1/2})$. For future convenience, let us define the following constants
\begin{align*}
C_{(A,B)}^{\det,\theta,p,\a}(s,T)
&:=\|\Sol_{s,(A,B)}(0,f,0)\|_{L^p((s,T)\times\Omega,w_{\a}^s;X_0)\to
L^p(\O;H^{\theta,p}(s,T,w_{\a}^s;X_{1-\theta}))}\\
C_{(A,B)}^{\stoc,\theta,p,\a}(s,T)
&:=\|\Sol_{s,(A,B)}(0,0,g)\|_{L^p((s,T)\times\Omega,w_{\a}^s;X_0)\to
L^p(\O;H^{\theta,p}(s,T,w_{\a}^s;X_{1-\theta}))}.
\end{align*}
Finally, for $\ell\in \{\deter,\stoc\}$, we set
\begin{align}\label{eq:constKellthetapa}
K^{\ell,\theta,p,\a}_{(A,B)}(s,T):=C_{(A,B)}^{\ell,\theta,p,\a}(s,T)+C_{(A,B)}^{\ell,0,p,\a}(s,T).
\end{align}

For later reference we collect some of the fundamental results on stochastic maximal $L^p$-regularity which will play a key role in the later subsections.

A sufficient condition for stochastic maximal $L^p$-regularity was obtained in \cite{MaximalLpregularity} for $\a=0$. The case $\a\in (0,\frac{p}{2}-1)$ was obtained in \cite{AV19} using a perturbation argument. For a complete weighted theory we refer to \cite{LV18}.

\begin{theorem}
\label{t:SMR_H_infinite}
Let the Assumption \ref{ass:X} be satisfied. Let $X_0$ be isomorphic to a closed subspace of an $L^q$-space for some $q\in [2,\infty)$, and let $\AA$ be a closed operator with domain $\Do(\AA)=X_1$. Assume that there exists $\lambda\in \R$ such that $\lambda
+\AA$ has a bounded $H^{\infty}$-calculus of angle $<\pi/2$. Then, for all $0\leq s<T<\infty$, $\AA\in \MRtas$.
If additionally $\AA$ is invertible, then $\sup_{T>s} K^{\ell, \theta,p,\a}_{\AA}(s,T)<\infty$ for $\ell\in \{\det,\stoc\}$.
\end{theorem}

The following results will be employed several times and has been shown in
\cite[Proposition 3.12 and Lemma 3.13]{AV19_QSEE_1}, where one could replace $\MRtas$ by $\MRtasz$.

\begin{proposition}
\label{prop:causality}
Let $(A,B)\in \MRtas$ and
\begin{equation*}
f\in  L^p_{\Progress}( (s,T)\times\Omega,w_{\a}^{s};X_0),
\quad\text{ and }\quad
g\in L^p_{\Progress}((s,T)\times\Omega,w_{\a}^s;\g(H,X_{1/2})),
\end{equation*}
and set $u=\Sol_{s,(A,B)}(0,f,g)$. Then for each stopping time $\tau$ such that $s\leq \tau\leq T$ a.s.\ and any strong solution $v\in L^p_{\Progress}((s,\tau)\times \O,w_{\a};X_1)$ to \eqref{eq:diffAB_s} the following holds.
\begin{enumerate}[{\rm(1)}]
\item\label{it:causality}
 $v$ can be written as
\begin{equation*}
v=u|_{[s,\tau]\times \Omega}=\Sol_{s,(A,B)}(0,\one_{[s,\tau]}f,\one_{[s,\tau]}g),\qquad \text{ on }[s,\tau]\times\Omega.
\end{equation*}
\item
\label{it:smallness_C_T}
There exists $C_T>0$ independent of $(f,g)$ such that $\lim_{T\downarrow 0}C_T=0$ and
$$
\|v\|_{L^p((s,\tau)\times \O,w_{\a}^s;X_0)}\leq
 C_T(\|f\|_{L^p((s,\tau)\times \O,w_{\a}^s;X_0)}+
 \|g\|_{L^p((s,\tau)\times \O,w_{\a}^s;\g(H,X_{1/2}))}).
$$
\end{enumerate}
\end{proposition}

A simple but useful transference result (see \cite[Proposition 3.8]{AV19_QSEE_1}) shows that in most cases $(A,B)\in \MRtasz$ self-improves to the stronger version with $\bullet$. To state the result we use the notation introduced in Assumption \ref{ass:AB_boundedness}.
\begin{proposition}[Transference]
\label{prop:transference}
Let $(A,B)\in \MRtasz$ and assume that there exists $(\wh{A},\wh{B})\in \MRtas$. Then $(A,B)\in \MRtas$ and
$$
C_{(A,B)}^{\ell,\theta,p,\a}(s,T)\leq C(C_{(\wh{A},\wh{B})}^{\ell,\theta,p,\a}(s,T),
C_{({A},{B})}^{\ell,0,p,\a}(s,T),C_{A,B},C_{\wh{A},\wh{B}})
$$
for all $\ell\in \{\deter,\stoc\}$ and $\theta\in [0,\frac{1}{2})\setminus \{\frac{1+\a}{p}\}$.
\end{proposition}

\section{A perturbation result}

Here and in the rest of this section, for the sake of brevity, for any $\theta\in [0,1]$ and any stopping time $\tau:\O\to [s,T]$, we set
\begin{equation}
\label{eq:E_theta_tau_spaces}
\begin{aligned}
E_{\theta,\kappa}(s,\tau) &= L^p_{\Progress}((s,\tau)\times\Omega,w_{\a}^{s};X_\theta), \\
E_{\theta,\kappa}^{\gamma}(s,\tau)  &= L^p_{\Progress}((s,\tau)\times\Omega,w_{\a}^{s};\gamma(H,X_\theta)).
\end{aligned}
\end{equation}

Before we state and prove our main perturbation result we provide a simple result which allows to reduce the question whether $(A,B)\in \MRtas$ to the same question on a finite partition of $(s,T)$.
\begin{proposition}[Partitions]
\label{prop:partitioninterval}
Let Assumptions \ref{ass:X} and \ref{ass:AB_boundedness} be satisfied. Assume that there exists $(\wh{A},\wh{B})\in \MRtas$ and fix $\delta\in (\frac{1+\a}{p},\frac{1}{2})$ if $p>2$, and any $\delta\in (0,1/2)$ if $p=2$.
Let $s=s_0<s_1<\dots<s_{N-1}<s_N=T$. Assume that for all $j\in\{1,\dots,N-1\}$,
$$
(A,B)\in \MRtasonesequencez,\qquad\quad (A,B)\in \MRtasjsequencez,
$$
and let $M>0$ be such that
$$\max\big\{C^{\deter,0,p,\a_j}_{(A,B)}(s_j,s_{j+1}),C^{\stoc,0,p,\a_j}_{(A,B)}(s_j,s_{j+1})\big\}\leq M$$
where $\a_0:=\a$ and $\a_j:=0$ if $j\geq 1$.
Then $(A,B)\in \MRtas$ and
\begin{equation}
\label{eq:max_reg_partition}
C^{\ell,0,p,\a}_{(A,B)}(s,T)
\leq C(N,M,(s_j)_{j=1}^N,
C_{(\wh{A},\wh{B})}^{\ell,\delta,p,\a}(s,T),
C_{\wh{A},\wh{B}}
),
\end{equation}
for $\ell\in \{\deter,\stoc\}$.
\end{proposition}

\begin{proof}
We only prove the case $p>2$, since the case $p=2$ is simpler. As usual, we only consider $s=0$. Let us recall that by \cite[Proposition 3.12]{AV19_QSEE_2},
\begin{equation}
\label{eq:AB_hat_t_cut}
(\wh{A},\wh{B})\in \mathcal{SMR}_{p,\a}^{\bullet}(0,T)\subseteq \mathcal{SMR}_{p,\a}^{\bullet}(0,t)\cap \MRtt
\end{equation}
for all $t\in [0,T]$. Thus, by Proposition \ref{prop:transference}, it is enough to show that if $(A,B)\in \mathcal{SMR}^{\bullet}_{p,\a}(0,s_{n-1})$ for a given $2\leq n\leq N$, then $(A,B)\in \mathcal{SMR}_{p,\a}(0,s_{n})$ and that \eqref{eq:max_reg_partition} holds with $s=0$ and $T=s_{n}$.
We content ourself to construct a unique strong solution $u_n$ to \eqref{eq:diffAB_s} on $[0,s_{n}]$ with a corresponding estimate. The fact that each strong solution $v$ to \eqref{eq:diffAB_s} on $[0,\tau]$ where $\tau$ is a stopping time such that $ 0\leq \tau\leq s_{n}$, satisfies $v=u_n|_{[0,\tau]}$ follows analogously.

Consider the problem \eqref{eq:diffAB_s} on $[0,s_n]$ with (see \eqref{eq:E_theta_tau_spaces})
\[u_{0} = 0, \qquad f\in E_{0,\kappa}(0,s_n), \qquad \text{and} \qquad
g\in  E_{1/2,\kappa}^{\g}(0,s_n).\]
Since $(A,B)\in \mathcal{SMR}_{p,\kappa}^{\bullet}(0,s_{n-1})$, and \cite[Theorem 1.2]{ALV20}, the right regularity holds at time $s_{n-1}$, and thus there exists a unique strong solution $u_{n-1}$ to \eqref{eq:diffAB_s} on $[0, s_{n-1}]$ such that
\begin{equation}\label{eq:uLpssn}
\begin{aligned}
\|u_{n-1}\|_{E_{1,\kappa}(0,s_{n-1})}+&
\|u_{n-1}\|_{L^p(\O;C([0,s_{n-1}];\Xp))}
\\&  \lesssim \|f\|_{E_{0,\kappa}(0,s_{n-1})}+
\|g\|_{E_{1/2,\kappa}^{\g}(0,s_{n-1})},
\end{aligned}
\end{equation}
where the implicit constant depends on $s_1,n,(\wh{A},\wh{B}),C^{\deter,0,p,\a}_{(A,B)}(0,s_{n-1}),C^{\stoc,0,p,\a}_{(A,B)}(0,s_{n-1})$. By assumption, Proposition \ref{prop:transference} and \eqref{eq:AB_hat_t_cut}, one has $(A,B)\in \mathcal{SMR}_{p}^{\bullet}(s_{n-1},s_{n})$. Thus, by \cite[Proposition 3.10]{AV19_QSEE_1} we can consider nonzero initial values and by \eqref{eq:uLpssn}, there exists a unique strong solution $U_n$ to the problem
\begin{equation*}
\begin{cases}
\dd U_n(t) +A(t)U_n(t)\, \dd t=f(t) \, \dd t+ (B(t)U_n(t)+g(t))\, \dd W_H(t),\quad t\in [s,s_{n}],\\
U_n(s_{n-1})=u_{n-1}(s_{n-1}),
\end{cases}
\end{equation*}
and
\begin{equation}
\begin{aligned}
\label{eq:uLpssnn}
\|U_n& \|_{E_{1,0}(s_{n-1}, s_n)}
\\& \lesssim
\|u_{n-1}(s_{n-1})\|_{L^p(\O;\Xp)}+
   \|f\|_{E_{0,0}(s_{n-1}, s_n)}+
\|g\|_{E_{1/2,0}^{\gamma}(s_{n-1}, s_n)}\\
&\lesssim
   \|f\|_{E_{0,\kappa}(0, s_n)}+
\|g\|_{E_{1/2,\kappa}^{\gamma}(0, s_n)}.
\end{aligned}
\end{equation}
Setting $u_n:=u_{n-1}$ on $[0,s_{n-1}]$ and $u_n:=U_n$ on $(s_{n-1},s_n]$, it follows that $u_n$ is a strong solution to \eqref{eq:diffAB_s}, and \eqref{eq:max_reg_partition} follows from \eqref{eq:uLpssn}-\eqref{eq:uLpssnn}.
\end{proof}

Let us conclude with a perturbation result which was announced in \cite{AV19_QSEE_1}. A version with random initial times, but without lower order terms (i.e.\ $L_A=L_B = 0$) was given in \cite{AV19_QSEE_2}. The perturbation result will be used in Section \ref{sec:SMRsecond} to treat $x$-dependent coefficients.

\begin{theorem}[Perturbation]
\label{t:pertubation}
Let Assumptions \ref{ass:X} and \ref{ass:AB_boundedness} be satisfied.
Assume that $(A,B)\in \mathcal{SMR}^{\bullet}_{p,\kappa}(s,T)\cap \mathcal{SMR}_{p}(s,T)$.
Fix $\delta\in (\frac{1+\kappa}{p},\frac{1}{2})$ if $p>2$, and any $\delta\in (0,\frac{1}{2})$ if $p=2$. Let $A_0:[s,T]\times \Omega\to{\mathscr L}(X_1,X_0)$, $B_0:[s,T]\times \Omega\to {\mathscr L}(X_1,\gamma(H,X_{1/2}))$ be strongly progressively measurable such that, for some positive constants $C_A,C_B\in (0,1)$, $L_A,L_B\in (0,\infty)$ and a.s.\ for all $x\in X_1$, $t\in (s,T)$,
\begin{align*}
\|A_0(t,\cdot)x\|_{X_0}&\leq C_{A} \|x\|_{X_1}+L_A\|x\|_{X_0}, \\
\|B_0(t,\cdot)x\|_{\g(H,X_{1/2})}&\leq C_{B} \|x\|_{X_1}+L_B\|x\|_{X_0}.
\end{align*}
Suppose that, for all $\mu\in \{0,\kappa\}$,
\begin{equation}
\label{eq:smallness_condition_perturbation}
 C^{\mathrm{det},0,p,\mu}_{(A,B)}(s,T)C_A+C^{\mathrm{sto},0,p,\mu}_{(A,B)}(s,T) C_B<1.
\end{equation}
Then $(A+A_0,B+B_0)\in \mathcal{SMR}^{\bullet}_{p,\kappa}(s,T)$ and, for $\ell\in \{\mathrm{det},\mathrm{sto}\}$,
\begin{align*}
C_{(A+A_0,B+B_0)}^{\ell,0,p,\kappa}(s,T)
\leq C\Big(p,\kappa,X_0,X_1,C_{A,B},C_A,C_B,L_A,L_B,\Big\{K^{\ell,\delta,p,\mu}_{(A,B)}(s,T)\Big\}_{{\ell\in \{\deter,\stoc\},\ \mu\in \{0,\kappa\}}}\Big).
\end{align*}
\end{theorem}

The key point in the above is the independence of the smallness condition \eqref{eq:smallness_condition_perturbation} on $L_A$ and $L_B$.
If one only assumes $(A,B)\in \mathcal{SMR}^{\bullet}_{p,\kappa}(s,T)$, then the proof below yields $(A+A_0,B+B_0)\in \mathcal{SMR}^{\bullet}_{p,\kappa}(s,t) $ for some $t<T$ provided \eqref{eq:smallness_condition_perturbation} holds for $\ell=\kappa$.

\begin{proof}
As usual, we set $s=0$. It suffices to prove the result on a suitable partition by Proposition \ref{prop:partitioninterval}. The proof is divided into three steps.

\textit{Step 1: There exists $s_1,C_1>0$ depending only on $p$, $\a$, $X_0$, $X_1$, $C_A$, $C_B$, $L_A$, $L_B$, $C_{A,B}$, $C^{\mathrm{det},0,p,\kappa}_{(A,B)}(0,T)$,  $C^{\mathrm{sto},0,p,\kappa}_{(A,B)}(0,T)$ such that $(A+A_0,B+B_0)\in \MRtasonesequencezer$ and}
$$
C_{(A+A_0,B+B_0)}^{\det,0,p,\a}(0,s_1) + C_{(A+A_0,B+B_0)}^{\stoc,0,p,\a}(0,s_1)\leq C_1.
$$
Fix $t\in [0,T]$, and let $\tau:\O\to [0,t]$ be a stopping time and let $E_{\theta,\kappa}(0,\tau)$ and $E_{\theta,\kappa}^{\g}(0,\tau)$ be as in \eqref{eq:E_theta_tau_spaces}. To prove the required result we use a stochastic version of the method of continuity (see \cite[Proposition 3.13]{AV19_QSEE_2}). To this end, for $\lambda\in [0,1]$, set $A^{(\lambda)}:=A+\lambda A_0$ and $B^{(\lambda)}:=B+\lambda B_0$. By \eqref{eq:smallness_condition_perturbation} with $\mu=\kappa$, we have
$$\eta:=1-C^{\det,0,p,\a}_{(A,B)}(0,T)C_A-C^{\stoc,0,p,\a}_{(A,B)}(0,T)C_B>0.$$
Let $\Sol:=\Sol_{0,(A,B)}$ be the solution operator associated to $(A,B)$. With the above choice of $A^{(\lambda)}, B^{(\lambda)}$ and Proposition \ref{prop:causality}, any strong solution $u\in L^p_{\Progress}((0,\tau)\times \Omega,w_{\a};X_1)$ to
\begin{equation*}
\begin{cases}
 \dd u(t) +A^{(\lambda)}u\, \dd t=f \, \dd t+ (B^{(\lambda)}u+g)\, \dd W_H,\quad \text{on} \ [0,\tau],\\
u(0)=0,
\end{cases}
\end{equation*}
satisfies
$$
u=\Sol(0,\one_{[0,\tau]}(f-\lambda A_0 u),\one_{[0,\tau]} (g+\lambda B_0 u)),\qquad \text{ on }[0,\tau]\times \Omega.
$$
For notational convenience, we set $v:=\Sol(0,\one_{[0,\tau]}(f-\lambda A_0 u),\one_{[0,\tau]} (g+\lambda B_0 u))$ on $[0,t]\times \Omega$. Thus $v|_{[0,\tau]\times \Omega}=u$. Using stochastic maximal $L^p$-regularity and $C^{\ell,0,p,\a}_{(A,B)}(0,t)\leq C^{\ell,0,p,\a}_{(A,B)}(0,T)$ for $t\leq T$ and $\ell\in \{\deter,\stoc\}$, we obtain
\begin{equation}
\label{eq:perturbation_step_1_1}
\begin{aligned}
\|u &\|_{E_{1,\kappa}(0,\tau)}
\leq\|v\|_{E_{1,\kappa}(0,t)}\\
&\leq C^{\det,0,p,\a}_{(A,B)}(0,T)\|f-\lambda A_0 u\|_{E_{0,\kappa}(0,\tau)}
+C^{\stoc,0,p,\a}_{(A,B)}(0,T)\|g+\lambda B_0 u\|_{E_{1/2,\kappa}^{\g}(0,\tau)}\\
&\leq (1-\eta)\| u\|_{E_{1,\kappa}(0,\tau)}+
\big[C^{\det,0,p,\a}_{(A,B)}(0,T)L_A+C^{\stoc,0,p,\a}_{(A,B)}(0,T)L_B\big]\| u\|_{E_{0,\kappa}(0,\tau)}\\
&\quad +C^{\det,0,p,\a}_{(A,B)}(0,T)\|f\|_{E_{0,\kappa}(0,\tau)}+C^{\stoc,0,p,\a}_{(A,B)}(0,T)\|g\|_{E_{1/2,\kappa}^{\g}(0,\tau)}.
\end{aligned}
\end{equation}
To estimate $\|u\|_{E_{0,\kappa}(0,\tau)}$ note that, by \cite[Lemma 3.13]{AV19_QSEE_1},
\begin{equation}
\begin{aligned}
\label{eq:perturbation_step_1_0}
\|u\|_{E_{0,\kappa}(0,\tau)}
&\leq c(t)\big(\|f-\lambda A_0 u\|_{E_{0,\kappa}(0,\tau)}+\|g+\lambda B_0 u\|_{E_{1/2,\kappa}^{\g}(0,\tau)}\big)\\
&\leq C(t)\big(\|u\|_{E_{1,\kappa}(0,\tau)}+\|f\|_{E_{0,\kappa}(0,\tau)}+\|g\|_{E_{1/2,\kappa}^{\g}(0,\tau)}\big)
\end{aligned}
\end{equation}
where $c(t)$ and $C(t)$ depend only on $p$, $\kappa$, $X_0$, $X_1$, $C_{A,B}$, $L_A$, $L_B$ and satisfy $\lim_{t\downarrow 0}c(t)=\lim_{t\downarrow 0}C(t)=0$ (here we also used that $C_A,C_B<1$). Next, we choose $s_1>0$ such that
$$
\big[C^{\det,0,p,\a}_{(A,B)}(0,T)L_A+C^{\stoc,0,p,\a}_{(A,B)}(0,T)L_B\big]C(s_1)<\frac{\eta}{2},
$$
Combining \eqref{eq:perturbation_step_1_1} and \eqref{eq:perturbation_step_1_0} we find the a priori estimate
\[\frac{\eta}{2} \|u\|_{E_{1,\kappa}(0,s_1)}\leq \Big(C^{\det,0,p,\a}_{(A,B)}(0,T)+\frac{\eta}{2}\Big)\|f\|_{E_{0,\kappa}(0,\tau)} + \Big(C^{\stoc,0,p,\a}_{(A,B)}(0,T)+\frac{\eta}{2}\Big)\|g\|_{E_{1/2,\kappa}^{\g}(0,\tau)}.\]
As the latter is uniform in $\lambda\in [0,1]$, the result follows from the method of continuity (see \cite[Proposition 3.13]{AV19_QSEE_2}).

\textit{Step 2: There exist $s',C'>0$ depending only on $p$, $X_0$, $X_1$, $C_A$, $C_B$, $L_A$, $L_B$, $C_{A,B}$, $C^{\mathrm{det},0,p,0}_{(A,B)}(0,T)$ $C^{\mathrm{sto},0,p,0}_{(A,B)}(0,T)$ such that, for each $t\in [0,T)$, one has $(A,B)\in \mathcal{SMR}_{p}(t,t')$ with $t':=\min\{t+s',T\}$ and
$$
\max\big\{C_{(A+A_0,B+B_0)}^{\det,0,p,0}(t,t'), C_{(A+A_0,B+B_0)}^{\stoc,0,p,0}(t,t')\big\}\leq C'.
$$ }
Note that, by \cite[Proposition 3.12]{AV19_QSEE_2} and $(A,B)\in \mathcal{SMR}_{p}(0,T)$, for all $\ell\in \{\mathrm{det},\mathrm{sto}\}$ and all $t'\in (t,T]$,
\begin{equation*}
C_{(A,B)}^{\ell,0,p,0}(t,t') \leq 
C_{(A,B)}^{\ell,0,p,0}(0,T).
\end{equation*}
Then \eqref{eq:smallness_condition_perturbation} for $\mu=0$ and the previous estimate gives
\begin{equation}
\label{eq:bound_constant_SMR_perturbation}
C_{(A,B)}^{\deter,0,p,0}(t,t')C_A+
C_{(A,B)}^{\stoc,0,p,0}(t,t')C_B <1.
\end{equation}
Up to a translation argument,
the claim of this step follows by repeating the argument in Step 1 with $\a=0$ and \eqref{eq:bound_constant_SMR_perturbation}.

\textit{Step 3: Conclusion}. Let $s_1$ and $s'>0$ be as in Step 1 and 2, respectively. The claim follows from Steps 1-2 and Proposition \ref{prop:partitioninterval} by setting $(\wh{A},\wh{B})=(A,B)$, $s_j:=\min\{s_1+js',T\}$ for $j\geq 2$ and choosing $N\in \N$ such that $s_1+N s'>T$.
\end{proof}

\section{Pointwise multiplication}
For the proofs in Section \ref{sec:SMRsecond} we need a result on pointwise multiplication in Bessel potential spaces. A standard reference for such results is \cite{RuSi}. Some extensions to the vector-valued setting can be found in \cite[Theorem 5.10]{MV15_multiplication}. Some of the results below are well-known (see \cite[Proposition 1.1, Chapter 2]{ToolsPDEs}). Since we could not find the all assertions we need in the literature, we include some details of the proof.
\begin{proposition}
\label{prop:multiplication_negative}
Let $H$ a Hilbert space, and let $\Dom$ either be $\R^d$, $\R^d_+$, $\T^d$, or a bounded Lipschitz domain. Let $s>0$, $q\in (1,\infty)$ and $\tau\in (s, \infty)$. Then the following estimates hold whenever the right-hand side is finite:
\begin{enumerate}[{\rm(1)}]
\item\label{it:multiplication_negative1}
$\|f g\|_{H^{s,q}(\Dom;H)}\lesssim
\|f\|_{H^{s,q_1}(\Dom)}\|g\|_{L^{q_2}(\Dom;H)}+
\|g\|_{H^{s,r_1}(\Dom;H)}\|f\|_{L^{r_2}(\Dom)}$
provided $q_1, r_1\in (1, \infty)$ and $q_2,r_2\in (1,\infty]$ satisfy $\frac{1}{q_1}+\frac{1}{q_2} = \frac{1}{r_1}+\frac{1}{r_2}=\frac{1}{q}$;
\item\label{it:multiplication_negative2}
$\|f g\|_{H^{s,q}(\Dom;H)}\lesssim
\|f\|_{H^{s,q}(\Dom)}\|g\|_{L^{\infty}(\Dom;H)}+
\|g\|_{C^{\tau}(\Dom;H)}\|f\|_{L^q(\Dom)}$;
\item\label{it:multiplication_negative3}
$
\displaystyle{
\|f g\|_{H^{-s,q}(\Dom;H)}\lesssim
\|f\|_{H^{-s,q}(\Dom)}\|g\|_{L^{\infty}(\Dom;H)}+
\|g\|_{H^{\tau,\zeta}(\Dom;H)}\|f\|_{H^{-s-\varepsilon,q}(\Dom)}
}
$
provided $\tau>\frac{d}{\zeta}$ and $\zeta\in [q',\infty)$ and where $\varepsilon>0$ only depends on $(d,q,s,\tau,\zeta)$;
\item\label{it:multiplication_negative4}
$
\displaystyle{
\|f g\|_{H^{-s,q}(\Dom;H)}\lesssim
\|f\|_{H^{-s,q}(\Dom)}\|g\|_{L^{\infty}(\Dom;H)}+
\|g\|_{C^{\tau}(\Dom;H)}\|f\|_{H^{-s-\varepsilon,q}(\Dom)},
}
$
where $\varepsilon\in (0,\tau-s)$ is arbitrary.
\end{enumerate}
\end{proposition}
The point in \eqref{it:multiplication_negative3} and \eqref{it:multiplication_negative4} is to find some $\varepsilon>0$ for which the stated estimate holds. Later on this will be enough to treat the term $\|f\|_{H^{-s-\varepsilon,q}(\Dom)}$ as a lower order perturbation in localization arguments. The admissible values of $\varepsilon$ can be obtained from the proof below. Finally, we note that one can also take $f$ to be $H$-valued and $g$ scalar-valued in the result and proof below. Moreover, one can even take $f$ and $g$ both $H$-valued if one replaces $fg$ by the inner-product $(f,g)_H$.
\begin{proof}
Since there exists a universal extension operator for $\Dom$, it suffices to consider $\Dom = \R^d$. For details on extension operators we refer to \cite[VI.3]{Stein70} in the integer case. The non-integer case can be obtained by complex interpolation or by using \cite{Rychkov1999} and the standard fact that $F^{s}_{p,2} = H^{s,p}$. All these results extend to the Hilbert space-valued setting with only minor modifications. In the proof below we use the notations of \cite{ToolsPDEs} for Bony's paraproducts. Let $(\psi_j)_{j\geq 0}$ be a Littlewood-Paley partition of the unity \cite[p.\ 4]{ToolsPDEs} and, for any $k\in\N$, set $\Psi_k:=\sum_{j=0}^k \psi_j$. Then $\Psi_k = \Psi_0(2^{-k}\cdot)$. For any $f\in \Sc'$, set $\psi_j(D)f:=\Fou^{-1}(\psi_j(\cdot)\Fou(f))$ where $\Fou$ denotes the Fourier transform on $\R^d$. Then, for any $f,g\in \Sc'$, Bony's decomposition of the product is given by $fg:=T_f g + R(f,g)+T_g f$, where
$$
T_f g:=\sum_{k\geq 5}\Psi_{k-5}(D)f\psi_{k+1}(D)g,\qquad
R(f,g)=\sum_{|j-k|\leq 4} \psi_j(D)f \psi_k(D) g,
$$
whenever these series converge in $\Sc'$.
The operator $T_f g$ is called Bony's paraproduct. The above extends to the case either $f$ or $g$ is in $\Sc'(\R^d;H)$. Moreover, by replacing the product by inner products in $H$ one can also define $(f,g)_H$ in $\Sc'(\R^d)$ whenever the appropriate series converge in $\Sc'(\R^d)$. All result below have versions for these situations, and we will need this in the duality argument in the proofs of \eqref{it:multiplication_negative3} and \eqref{it:multiplication_negative4}.

\eqref{it:multiplication_negative1}: This follows from \cite[Proposition 1.1, Chapter 2]{ToolsPDEs}. Actually only the case $H = \R$ is considered there, but the extension to the Hilbert space-valued setting is straightforward. In particular, note that Littlewood--Paley type estimates also hold in this setting, replacing the absolute values in the square functions by the norms $\|\cdot\|_H$.

\eqref{it:multiplication_negative2}: This can be proved in a similar way as \eqref{it:multiplication_negative1}. It suffices to bound each of the terms in the paraproduct. To indicate the required change we recall from (see \cite[Theorem 2.5.7(6)]{Tri83} and \cite[Theorem 2.7.2.1]{Tr1}) that
\begin{align}\label{eq:ZygmundHolder}
\|g\|_{C^{t}}\eqsim \sup_{j\geq 0}2^{jt}\|\psi_j(D)g\|_{L^{\infty}}, \ \ \text{for all $t\in (0,\infty)\setminus \N$}.
\end{align}
By (1.4)-(1.5) in \cite[Chapter 2]{ToolsPDEs},
\begin{equation}
\label{eq:R_T_g_estimates}
\|R(f,g)\|_{H^{s,q}}+\|T_g f\|_{H^{s,q}}\lesssim \|g\|_{L^{q_1}}\|f\|_{H^{s,q_2}}\ \ \text{ where } \ \ \frac{1}{q_1}+\frac{1}{q_2}=\frac{1}{q},
\end{equation}
with $q_1\in (q, \infty]$ and $q_2\in (q,\infty)$. Setting $q_1 = \infty$ part of  \eqref{it:multiplication_negative2} follows and it remains to estimate $T_f g$.
Reasoning as in \cite[(1.6), Chapter 2]{ToolsPDEs}, one obtains that for all $\gamma>\varepsilon\geq 0$ such that $s+\gamma\notin \N$ (in order to apply \eqref{eq:ZygmundHolder})
\begin{equation}
\begin{aligned}
\label{eq:T_g_estimate_varepsilon}
\|T_f g\|_{H^{s+\varepsilon,q}}
&\eqsim \Big\|\Big(\sum_{k\geq 5} 2^{2k(s+\varepsilon) } |\Psi_{k-5}(D) f|^2 \|\psi_k(D) g\|_H^{2}\Big)^{1/2}\Big\|_{L^q}
\\ & \leq \Big\|\sup_{k\geq 0}2^{(s+\gamma)k}\|\psi_k(D) g\|_H\Big\|_{L^\infty} \Big\|\Big(\sum_{k\geq 5} 2^{-2k(\gamma-\varepsilon)} |\Psi_{k-5}(D) f|^2 \Big)^{1/2}\Big\|_{L^q}
\\
&\lesssim \|g\|_{C^{s+\gamma}}\Big\|\sup_{k\geq 0}|\Psi_{k}(D) f|\Big\|_{L^q}\\
&\lesssim \|g\|_{C^{s+\gamma}} \|f\|_{L^q} ,
\end{aligned}
\end{equation}
where in the last inequality we used \cite[Proposition 2.3.9 and Theorem 2.3.2]{Analysis1}. With $\varepsilon=0$ and $s+\gamma\in (s,\tau]\setminus\N$, the remaining estimate for \eqref{it:multiplication_negative2} follows.

\eqref{it:multiplication_negative3}: By Sobolev embedding we may assume $\zeta\in (q',\infty)$. We use a duality argument.
Note that $\Sc(\R^d;H)$ is dense in $H^{t,q}$ and $(H^{-t,q})^*=H^{t,q'}$ for all $t\in \R$ and $q\in (1, \infty)$ (see \cite[Propositions 5.6.4 and 5.6.7]{Analysis1} for the vector-valued case).
By approximation it suffices to consider $f\in \Sc(\R^d)$.
Let $h\in \Sc(\R^d;H)$ and $g\in H^{\tau,\zeta}(\R^d;H)$ and note that $H^{\tau,\zeta}\embed L^{\infty}$ by Sobolev embeddings. Bony's decomposition yields
\begin{equation}
\label{eq:claim_T_h_g_Step_1_0_0}
|\l fg ,h\r|=|\l f,(g,h)_H\r|\leq |\l f, T_g h\r|+|\l f,R(g,h)\r|+ |\l f, T_h g\r|,
\end{equation}
where we need the analogues of $T_g, T_h$ and $R$ where inner products are taken into account. By applying \eqref{eq:R_T_g_estimates} we find
\begin{equation}
\label{eq:claim_T_h_g_Step_1_0}
\begin{aligned}
|\l f, T_g h\r|+|\l f,R(g,h)\r|
&\leq \|f\|_{H^{-s,q}}\big(\|T_g h\|_{H^{s,q'}}+\|R(g,h)\|_{H^{s,q'}}\big)\\
&\lesssim \|f\|_{H^{-s,q}}\|g\|_{L^{\infty}(\R^d;H)}\| h\|_{H^{s,q'}(\R^d;H)}.
\end{aligned}
\end{equation}
Since
$
|\l f, T_h g\r|\lesssim \|f\|_{H^{-s-\varepsilon,q}}\|T_h g\|_{H^{s+\varepsilon,q'}}
$
for all $\varepsilon>0$, it remains to prove
\begin{equation}
\label{eq:claim_T_h_g_Step_1}
\|T_h g\|_{H^{s+\varepsilon,q'}}\lesssim \|h\|_{H^{s,q'}(\R^d;H)}\|g\|_{H^{\tau,\zeta}(\R^d;H)}
\end{equation}
where $\varepsilon>0$ depends solely on $(d,q,s,\tau,\zeta)$. Indeed, if \eqref{eq:claim_T_h_g_Step_1} holds, then inserting \eqref{eq:claim_T_h_g_Step_1_0}-\eqref{eq:claim_T_h_g_Step_1} in \eqref{eq:claim_T_h_g_Step_1_0_0}, we obtain
$$
|\l fg ,h\r| \lesssim (\|f\|_{H^{-s,q}}\|g\|_{L^{\infty}(\R^d;H)}+\|f\|_{H^{-s-\varepsilon,q}}\|g\|_{H^{\tau,\zeta}(\R^d;H)})\|h\|_{H^{s,q'}(\R^d;H)}.
$$
By taking the supremum over all $h\in \Sc(\R^d; H)$ with $\|h\|_{H^{s,q'}(\R^d;H)}\leq 1$, the estimate in \eqref{it:multiplication_negative3} follows by density and duality.

To prove \eqref{eq:claim_T_h_g_Step_1} we will use the following estimation:
\begin{align*}
\|T_h g\|_{H^{s+\varepsilon,q'}}\stackrel{\eqref{eq:R_T_g_estimates}}{\lesssim} \|h\|_{L^{q_1}(\R^d;H)} \|g\|_{H^{s+\varepsilon,q_2}(\R^d;H)} \stackrel{(*)}{\lesssim} \|h\|_{H^{s,q'}(\R^d;H)} \|g\|_{H^{\tau,\zeta}(\R^d;H)}.
\end{align*}
Here we apply \eqref{eq:R_T_g_estimates} with $\frac{1}{q_1}+\frac{1}{q_2}=\frac{1}{q'}$, and in $(*)$ we used Sobolev embedding which gives the restrictions
\begin{align}\label{eq:restrictionnegpoint}
(i): \ s-\frac{d}{q'}\geq -\frac{d}{q_1}, \ \ \text{and} \ \ (ii): \ q_2>\zeta, \ \ \text{and} \ \ (iii): \ \tau-\frac{d}{\zeta} \geq s+\varepsilon - \frac{d}{q_2}.
\end{align}
Note that the first condition in \eqref{eq:restrictionnegpoint} is equivalent to $s\geq d/q_2$. Since we assumed $\zeta>q'$, we can always find $q_1$ such that $\frac{1}{q_1}+\frac{1}{q_2}=\frac{1}{q'}$ for any given $q_2>\zeta$. In order to find parameters $(\varepsilon,q_2)$ for which \eqref{eq:restrictionnegpoint} holds, we consider two cases.

\emph{Case $\zeta\geq d/s$}. Let $\varepsilon\in (0,\tau-s)$ be arbitrary and set $q_2 = \zeta+\delta$ with $\delta>0$ arbitrary but fixed and will be chosen below. Then (i) and (ii) of \eqref{eq:restrictionnegpoint} hold, and (iii) is equivalent to $\tau-s-\varepsilon\geq \frac{d}{\zeta}  - \frac{d}{\zeta+\delta}$ which holds if $\delta>0$ is chosen small enough.

\emph{Case $\zeta<d/s$}. Let $\varepsilon = \tau-\frac{d}{\zeta}>0$ and $q_2 = d/s$. Then (i), (ii) and (iii) of \eqref{eq:restrictionnegpoint} are clear.

\eqref{it:multiplication_negative4}: The proof follows as in the previous step, where the following variant of \eqref{eq:claim_T_h_g_Step_1} should be used:
\[\|T_h g\|_{H^{s+\varepsilon,q}(\R^d)}\lesssim \|g\|_{C^{\tau}}\|h\|_{L^{q'}(\R^d;H)}\lesssim \|g\|_{C^{\tau}(\R^d;H)}\|h\|_{H^{s,q'}(\R^d;H)},\]
where we also used \eqref{eq:T_g_estimate_varepsilon}.
\end{proof}

As a consequence of Proposition \ref{prop:multiplication_negative}\eqref{it:multiplication_negative1} we obtain:
\begin{corollary}
\label{cor:multiplication_Sobolev}
Let $H$ a Hilbert space, and let $\Dom$ either be $\R^d$, $\R^d_+$, $\T^d$, or a bounded Lipschitz domain. Let $s\in (0,\infty)$ and $q\in (1,\infty)$. Suppose that $\xi\in (1,q]$ and $\eta\geq s$ satisfy $\eta>\frac{d}{\xi}$, $\eta-\frac{d}{\xi}\geq s-\frac{d}{q}$. Then there exists $\varepsilon(d,\xi,\eta,\reg,q)>0$ such that
$$
\|f g\|_{H^{s,q}(\Dom;H)}
\lesssim
\|f\|_{H^{s,q}}\|g\|_{L^{\infty}(\Dom;H)} +
\|f\|_{H^{s-\varepsilon,q}}\|g\|_{H^{\eta,\xi}(\Dom;H)}
$$
whenever the right-hand side is finite.
\end{corollary}

\begin{proof}
As before it suffices to consider $\Dom = \R^d$.
Due to Proposition \ref{prop:multiplication_negative}\eqref{it:multiplication_negative1}
for all $\zeta\in (1,\infty)$ and $\varrho\in (1,\infty]$ such that $\frac{1}{q}=\frac{1}{\zeta}+\frac{1}{\varrho}$ one has
\begin{equation*}
\begin{aligned}
\|f g\|_{H^{s,q}(\R^d;H)}
\lesssim
\|f\|_{H^{s,q}}\|g\|_{L^{\infty}(\R^d;H)} +
\|f\|_{L^{\varrho}}\|g\|_{H^{s,\zeta}(\R^d;H)} .
\end{aligned}
\end{equation*}
Therefore, to prove the claim it remains to show that
\begin{equation}
\label{eq:embeddings_estimate_step_7}
H^{\eta,\xi}(\R^d;H) \hookrightarrow H^{s,\zeta}(\R^d;H)
\ \ \ \text{ and }\ \ \
H^{s-\varepsilon,q} \hookrightarrow L^{\varrho} \ \ \text{ for some }\varepsilon>0
\end{equation}
where $\varrho,\zeta$ will be suitably chosen. To prove \eqref{eq:embeddings_estimate_step_7} we apply Sobolev embedding and split into two cases.

\emph{Case $s-\frac{d}{q}>0$.} One can check that the choice $\varrho=\infty$, $\zeta=q$ and the assumption $\eta-\frac{d}{\xi}\geq s-\frac{d}{q}$ yield \eqref{eq:embeddings_estimate_step_7} for all $\varepsilon\in (0,s-\frac{d}{q})$.

\emph{Case $s-\frac{d}{q}\leq  0$}. Fix $\varepsilon\in (0,\min\{s,\eta-\frac{d}{\xi}\})$. Let $\varrho\in (q,\infty)$ be such that $s-\varepsilon-\frac{d}{q}=-\frac{d}{\varrho}$ (i.e.\ $s-\frac{d}{\zeta}=\varepsilon$). Thus the second embedding in \eqref{eq:embeddings_estimate_step_7} holds. To obtain the first one in \eqref{eq:embeddings_estimate_step_7} it is enough to note that $\eta-\frac{d}{\xi}>\varepsilon= s-\frac{d}{\zeta}$ and $\xi\leq q<\zeta$.
\end{proof}

\section{Main result}\label{sec:SMRsecond}

Below we gives sufficient conditions for stochastic maximal regularity for systems of second order operators on the torus $\T^d = [-1/2,1/2]^d$. It is the first result on $L^p((0,T)\times\Omega,H^{\sigma,q}(\Tor^d))$-theory with coefficients which are measurable in time and $\Omega$, and H\"older continuous in space. The flexibility to have $p,q\geq 2$ and $\sigma\in \R$ arbitrary, is highly relevant for the applications to nonlinear SPDEs (see \cite{AV19_QSEE_1,AV19_QSEE_2}).

\subsection{Assumptions}
Consider the following linear parabolic system on $\Tor^d$:
\begin{equation}
\label{eq:parabolic_problem}
\begin{cases}
\dd u +\A\, u \, \dd t= f \, \dd t + \sum_{n\geq 1} (\b_n u+g_n) \, \dd w^n_t,&\text{ on }\Tor^d,\\
u(s)=u_s, &\text{ on }\Tor^d,
\end{cases}
\end{equation}
where $s\in [0,T)$, $m,d\in \N$ and for each sufficiently smooth $u=(u_{k})_{k=1}^m$,
\begin{equation}
\label{eq:def_AB_parabolic}
\begin{aligned}
\A(t) u&:=-\sum_{i,j=1}^d \partial_i (\am^{i,j}(t,\cdot)\partial_ju)  \  \ \text{or}  \ \
\wt{\A}(t) u:=-\sum_{i,j=1}^d \am^{i,j}(t,\cdot)\partial_i \partial_ju,
\\ \b_n(t) u&:=\Big(\sum_{j=1}^d \bm^{j}_{k,n}(t,\cdot)\partial_ju_{k}\Big)_{k=1}^m
\end{aligned}
\end{equation}
for all $n\geq 1$, $t\in (s,T)$. Lower order terms can be added without difficulty afterwards by the perturbation results of Theorem \ref{t:pertubation}. Moreover, using a minor modification of \eqref{it4:parabolic_problems} below, complex coefficient can be allowed as well (see \cite[Assumption 5.2(2)]{VP18}). A natural question is whether the diagonal structure on the operator $\b$ can be relaxed. In general this is not the case as follows from a counterexample in \cite{KimLeesystems}. Finally, $(w_t^n:t\geq 0)_{n\geq 1}$ denotes a sequence of standard Brownian motions on a filtered probability $(\O,\mathscr{A},(\F_t)_{t\geq 0},\P)$. By \cite[Example 2.12]{AV19_QSEE_1} this naturally determines a $W_{\ell^2}$-cylindrical Brownian motion.

Consider the following assumption in the case of a divergence form operators.

\begin{assumption} Let $m,d\in\N$, $\sigma \in \R$ and $0\leq s<T<\infty$.
\label{ass:parabolic_problems}
\begin{enumerate}[{\rm(1)}]
\item Let one of the following be satisfied:
\begin{itemize}
\item $q\in [2,\infty)$, $p\in (2,\infty)$ and $\a\in [0,\frac{p}{2}-1)$;
\item $p=q=2$ and $\a=0$.
\end{itemize}
\item\label{it:parabolic_problems1} For each $n\geq 1$, $i,j\in \{1,\dots,d\}$ and $k\in \{1,\dots,m\}$, the maps $\am^{i,j}:=(\am^{i,j}_{k,h})_{k,h=1}^m:[0,T]\times \O\times \Tor^d\to \R^{m\times m}$, $\bm^{j}_{k,n}:[0,T]\times \O\times \Tor^d\to \R$ are $\Progress\otimes \Borel(\Tor^d)$-measurable.
\item\label{it:parabolic_problems2} There exist $C_{\am,\bm}>0$, $\alpha>|1+\sigma|$ such that a.s.\ for all $t\in (s,T)$, $i,j\in \{1,\dots,d\}$, $k\in \{1,\dots,m\}$,
\begin{align*}
\|\am^{i,j}(t,\cdot)\|_{C^{\alpha}(\Tor^d;\R^{m\times m})}+\|(\bm^{j}_{k,n}(t,\cdot))_{n\geq 1}\|_{C^{\alpha}(\Tor^d;\ell^2)}&\leq C_{\am,\bm}.
\end{align*}

\item\label{it4:parabolic_problems} Let $\Psi^{i,j}$ be the $m\times m$-dimensional diagonal matrix whose diagonal elements are given by $(\frac{1}{2}\sum_{n\geq 1} \bm^j_{\ell,n}\bm^i_{\ell,n})_{\ell=1}^m$. There exists $\vartheta>0$ such that a.s.\ for all $t\in [0,T]$, $x\in \Tor^d$, $\eta\in \R^m$ and $\xi\in \R^d$,
$$
\sum_{i,j=1}^d \big([(\am^{i,j}(t,x)-\Psi^{i,j}(t,x)) \eta]\cdot \eta\big) \xi_i \xi_j\geq \vartheta |\xi|^2|\eta|^2.
$$
\end{enumerate}
\end{assumption}
In the case $\A$ is in non-divergence form, the only required change is that instead of \eqref{it:parabolic_problems2} we assume

\let\ALTERWERTA\theenumi
\let\ALTERWERTB\labelenumi
\def\theenumi{(3)'}
\def\labelenumi{(3)'}
\begin{enumerate}
\itemindent=-13pt
\item\label{it:threeprime}
\emph{There exist $C_{\am,\bm}>0$, $\alpha>|1+\sigma|$ and $\beta>|\sigma|$ such that a.s.\ for all $t\in (s,T)$, $i,j\in \{1,\dots,d\}$, $\ell\in \{1,\dots,m\}$,
\begin{align*}
\|\am^{i,j}(t,\cdot)\|_{C^{\beta}(\Tor^d;\R^{m\times m})}+\|(\bm^{j}_{\ell,n}(t,\cdot))_{n
\geq 1}\|_{C^{\alpha}(\Tor^d;\ell^2)}&\leq C_{\am,\bm}.
\end{align*}
}
\end{enumerate}
\let\theenumi\ALTERWERTA
\let\labelenumi\ALTERWERTB

For notational convenience we write $H^{s,q}$ and $B^{s}_{q,p}$ instead of $H^{s,q}(\Tor^d;\R^m)$ and $B^{s}_{q,p}(\Tor^d;\R^m)$ below. Let
\begin{align}\label{eq:X0X1}
X_0 = H^{\sigma,q} \ \ \text{and} \ \ X_1 = H^{2+\sigma,q}.
\end{align}
Then \[X_{\theta} = H^{2\theta+\sigma} \ \  \text{for $\theta\in [0,1]$} \ \ \text{and} \  \  X_{\kappa,p}^{\Tr} = B^{2+\sigma-\frac{2(1+\kappa)}{p}}_{q,p}.\]
If $p=q=2$ and $\kappa=0$, and $X_{0,2}^{\Tr} = H^{1+\sigma,2}$.
For $(\A,\b)$ as in \eqref{eq:def_AB_parabolic}, we set
\begin{equation*}
\begin{aligned}
(A,B)&:[0,T]\times \O\to \calL(H^{2+\sigma,q},H^{\sigma,q}\times H^{1+\sigma,q}(\ell^2)),\\
(A,B)u&:=(\A u,(\b_n u)_{n\geq 1}),\qquad  u\in H^{2+\sigma,q},
\end{aligned}
\end{equation*}
and similarly for $(\wt{A},B)$. Here $H^{1+\sigma,q}(\ell^2):=H^{1+\sigma,q}(\T^d;\ell^2)$.
The couples $(A,B)$ and $(\wt{A},B)$ are well-defined and actually satisfy Assumption \ref{ass:AB_boundedness} with $C_{(A,B)}\leq C(q,p,\sigma,d,m)$.
Indeed, by Assumption \ref{ass:parabolic_problems} and Proposition \ref{prop:multiplication_negative}\eqref{it:multiplication_negative2},\eqref{it:multiplication_negative4} and \cite[Proposition 9.3.1]{Analysis2} it follows that for all $v\in H^{2+\sigma,q}$,
\begin{align*}
&\|\A v\|_{H^{\sigma,q}}+\|(\b_n v )_{n\geq 1}\|_{\g(\ell^2,H^{1+\sigma,q})}\\
&\lesssim
\max_{i,j}\| a^{i,j}\nabla v\|_{H^{1+\sigma,q}}+\|(\b_n v )_{n\geq 1}\|_{H^{1+\sigma,q}(\ell^2)}
\lesssim C_{a,b} \|v\|_{H^{2+\sigma,q}},
\end{align*}
where the spaces take values in $\R^m$. The required measurability follows from the progressively measurability of the coefficients $a^{i,j}$ and $b^j_{k,n}$, where for the $\gamma$-valued case one can use \cite[Lemma 2.5]{NVW1}.

\subsection{Statements of the main results}
Next we state our main results on stochastic maximal regularity for these operators. Equivalently we could just write $(A,B)\in \MRtas$ (see Definition \ref{def:SMRgeneralized}) below, and write the estimates by referring to the constants $K^{\theta}_{(A,B)}$ in \eqref{eq:constKellthetapa}. However, to make the theorems more accessible we write out the definitions. Note that Theorem \ref{t:parabolic_problems-intro} is a special case of the result below.
\begin{theorem}[Divergence form]
\label{t:parabolic_problems}
Suppose that Assumption \ref{ass:parabolic_problems} holds. Then for each $s\in [0,T)$, each $\F_s$-measurable $u_s\in L^p(\Omega;B^{2+\sigma-2(1+\kappa)/p}_{q,p})$, each progressively measurable $f\in  L^p(  (s ,T)\times \Omega,w_{\a}^{s};H^{\sigma,q})$, and
$g\in L^p( (s ,T)\times \Omega,w_{\a}^s;H^{1+\sigma,q}(\ell^2))$ there exists a unique strong solution $u$ to \eqref{eq:parabolic_problem}. Moreover, letting
\begin{align*}
  J_{p,q,\kappa}(u_s,f,g):=  \|u_s\|_{L^p(\Omega;B^{2+\sigma-2(1+\kappa)/p}_{q,p})}  +
   \|f\|_{L^p( (s ,T)\times \Omega,w_{\a}^{s};H^{\sigma,q})}
   +\|g\|_{L^p( (s ,T)\times \Omega,w_{\a}^{s};H^{1+\sigma,q}(\ell^2))},
\end{align*}
the following results hold, where the constants $C_i$ do not depend on $(s,u_s,f,g)$:
\begin{enumerate}[{\rm(1)}]
\item in case $p\in (2, \infty)$, $q\in [2, \infty)$, $\theta\in [0,\frac12)$,
\begin{align*}
\|u\|_{L^p((s,T)\times \Omega,w_{\a}^{s};H^{2+\sigma,q})} &\leq C_1 J_{p,q,\kappa}(u_s,f,g),
\\ \|u\|_{L^p(\O;\hz^{\theta,p}(s,T,w_{\a}^{s};H^{2-2\theta+\sigma,q}))} &\leq C_{2}(\theta) J_{p,q,\kappa}(u_s,f,g), \ \ \theta\in [0,1/2),
\\ \|u\|_{L^p(\O;C([s,T];B^{2+\sigma-2(1+\kappa)/p}_{q,p}))}
&\leq C_3 J_{p,q,\kappa}(u_s,f,g),
\\ \|u\|_{L^p(\O;C([s+\varepsilon,T];B^{2+\sigma-2/p}_{q,p}))}
&\leq C_{4}(\varepsilon) J_{p,q,\kappa}(u_s,f, g), \ \ \varepsilon\in (s,T).
\end{align*}
\item in case $p=q=2$ (and thus $\a=0$),
\begin{align*}
\|u\|_{L^2(\O;L^2(s,T;H^{2+\sigma,2}))} &\leq C_5 J_{2,2,0}(u_s,f,g),
\\ \|u\|_{L^2(\O;C([s,T];H^{1+\sigma,2}))} &\leq C_6 J_{2,2,0}(u_s,f,g).
\end{align*}
\end{enumerate}
\end{theorem}
From the proof it is clear that $C_i$ only depends on $q, p, \a, \sigma, d, m, \vartheta, C_{\am,\bm}, T$, and additionally on $\theta$ and $\varepsilon$ if this is written explicitly. One can use (weighted) Sobolev embedding to deduce H\"older regularity from the above theorem:
\[\hz^{\theta,p}(s,T;X)\hookrightarrow C^{\theta-\frac{1+\kappa}{p}}([s,T];X) \ \ \text{and} \ \ H^{\alpha,q}(\Dom)\hookrightarrow C^{\alpha-\frac{d}{q}}(\overline{\Dom})\]
if $\theta>\frac{1+\kappa}{p}$ (see \cite[Proposition 7.4]{MV12}), and if $\alpha-\frac{d}{q}\notin\N_0$, respectively.

The following variant holds in the non-divergence case:
\begin{theorem}[Non-divergence form]\label{t:parabolic_problems_non}
Suppose that Assumption \ref{ass:parabolic_problems} with \eqref{it:parabolic_problems2} replaced by \emph{\ref{it:threeprime}}. Then for each $s\in [0,T)$ the same assertions as in Theorem \ref{t:parabolic_problems} hold if $(\A,\b)$ is replaced by $(\wt{\A},\b)$.
\end{theorem}

In principle our proof is self-contained, but we do rely on the results for constant coefficients in space \cite[Theorem 5.3]{VP18}, which also holds in the case $\R^d$ is replaced by $\T^d$ as is clear from the proof.
Besides the fact that Theorems \ref{t:parabolic_problems} and \ref{t:parabolic_problems_non} are more general, the proof below is also simpler in the sense that we can completely avoid the case $\theta\neq 0$ in \cite[Lemmas 5.5-5.7 and step 2 of Theorem 5.4]{VP18} which was needed for the optimal mixed space-time smoothness. Instead we rely on the transference result of Proposition \ref{prop:transference}.

A key step in the proof of Theorem \ref{t:parabolic_problems} is the following a priori estimate on small time intervals.
\begin{lemma}
\label{l:estimates_small_interval}
Let Assumption \ref{ass:parabolic_problems} be satisfied.
Then there exist $\tT,C>0$ depending only on $q, p, \a, \sigma, d, m, \vartheta, C_{\am,\bm}$ for which the following holds:

For any $t\in [s,T]$, $\ell\in \{0,\a\}$, any stopping time $\tau:\O\to [t,T\wedge(t+\tT)]$, any $f\in L^p_{\Progress}((t,\tau)\times\Omega,w_{\ell}^t;X_0)$, $g\in L^p_{\Progress}((t,\tau)\times\Omega,w_{\ell}^t;\g(H,X_{1/2}))$ and any strong solution $u\in L^p_{\Progress}((t,\tau)\times\Omega,w_{\ell}^t;X_1)$ to \eqref{eq:parabolic_problem} on $[t,\tau]$ with $u(t)=0$ one has
\begin{align}\label{eq:lemaprioriTorus}
\|u\|_{L^p((t,\tau)\times\Omega,w_{\ell}^t;X_1)}\leq C (\|f\|_{L^p((t,\tau)\times\Omega,w_{\ell}^t;X_0)}+\|g\|_{L^p((t,\tau)\times\Omega,w_{\ell}^t;\g(H,X_{1/2}))}).
\end{align}
The same assertion holds in the non-divergence case if Assumption \ref{ass:parabolic_problems}\eqref{it:parabolic_problems2} is replaced by \emph{\ref{it:threeprime}}.
\end{lemma}

We first show how Theorems \ref{t:parabolic_problems} and \ref{t:parabolic_problems_non} can be derived from the above lemma.

\begin{proof}[Proof of Theorems \ref{t:parabolic_problems} and \ref{t:parabolic_problems_non}]
All assertions would follow if we can show that  $(A,B)\in \mathcal{SMR}_{p,\a}^{\bullet}(0,T)$. Indeed, it suffices to consider zero initial data $u_s=0$ by \cite[Proposition 3.10]{AV19_QSEE_1}. In that case all assertions follow from the definitions except the maximal inequalities which are a direct consequence of $(A,B)\in \mathcal{SMR}_{p,\a}^{\bullet}(0,T)$ and the trace embedding in \cite[Theorem 1.2]{ALV20}.

By the periodic version of \cite[Theorem 10.2.25]{Analysis2}, $1-\Delta: H^{2+\sigma,q}\subseteq H^{\sigma,q}\to H^{\sigma,q}$ has a bounded $H^{\infty}$-calculus of angle $<\pi/2$. Hence Theorem \ref{t:SMR_H_infinite} implies that $-\Delta\in \MRta$ for all $T<\infty$. Therefore, by Proposition \ref{prop:transference} it is enough to prove that $(A,B)\in \mathcal{SMR}_{p,\a}(0,T)$.

By Proposition \ref{prop:partitioninterval} it suffices to prove $(A,B)\in \mathcal{SMR}_{p,\ell}(t,t+T^*)$ for all $t\in [0,T-T^*]$, where $T^*$ is as in Lemma \ref{l:estimates_small_interval} and $\ell\in \{0,\a\}$.
To show this we will apply the method of continuity (see \cite[Proposition 3.13]{AV19_QSEE_2}). Therefore, fix a stopping time $\tau:\O\to [t+t\wedge T^*]$ and let $u\in L^p_{\Progress}((t,\tau)\times\Omega,w_{\ell}^t;X_1)$ be a strong solution to
\eqref{eq:parabolic_problem} with $(A,B)$  replaced by
\[A_{\lambda} u = \lambda A u + (1-\lambda)(-\Delta) u, \ \ \text{and} \ \ B_{n,\lambda} u = \lambda B_n u,\]
where $\lambda\in [0,1]$ is fixed. By the above mentioned method of continuity with $X_0$ and $X_1$ as in \eqref{eq:X0X1}, it suffices to
prove the a priori estimate \eqref{eq:lemaprioriTorus}.
Since $(A_{\lambda},B_{\lambda})$ satisfies Assumption \ref{ass:parabolic_problems} (uniformly in $\lambda$), the required a priori estimate follows from Lemma \ref{l:estimates_small_interval}, and thus the theorem follows.
\end{proof}

In the next remark we discuss several situations in which the regularity conditions on the coefficients $a$ and $b$ can be weakened, and under which the assertions of Theorems \ref{t:parabolic_problems} and \ref{t:parabolic_problems_non} still remain valid.
\begin{remark}
\label{r:less_regularity_integer_delta_case}
In special integer cases of $\sigma$ one can check from  Steps 2 and 5 of the proof of Lemma \ref{l:estimates_small_interval} that the H\"older condition of Assumption \ref{ass:parabolic_problems}\eqref{it:parabolic_problems2} can be weakened. The proof always requires the existence of a $\modulus\in C([0,\infty))$ such that $\modulus(0)=0$ and for all $t\in [s,T]$, $x,x'\in \Tor^d$, $i,j\in \{1,\dots,d\}$, $k\in \{1,\dots,m\}$,
\[\|\am^{i,j}(t,x)- \am^{i,j}(t,x')\|_{\R^{m\times m} } + \|(\bm_{k,n}^j(t,x)-\bm_{k,n}^j(t,x'))_{n\geq 1}\|_{\ell^2}\leq \modulus(|x-x'|).\]
Below we indicate which additional assumption are needed in particular cases:
\begin{itemize}
\item divergence case, $\sigma=-1$: there exist $C_{\am,\bm}>0$, $\alpha>0$, a.s.\ for all $t\in [0,T]$, $i,j\in \{1,\dots,d\}$, $k\in \{1,\dots,m\}$,
\begin{align*}
\|\am^{i,j}(t,\cdot)\|_{C^{\alpha}(\Tor^d;\R^{m\times m})}+
 \|(\bm_{k,n}^{j}(t,x))_{n\geq 1}\|_{\ell^2}
 &\leq C_{\am,\bm}.
\end{align*}
\item divergence case, $\sigma=0$: there exist $C_{\am,\bm}>0$, a.s.\ for all $t\in [0,T]$, $x,x'\in \Tor^d$, $i,j,k\in \{1,\dots,d\}$, $k\in \{1,\dots,m\}$, $r\in \{0,1\}$,
\begin{align*}
 \|\partial_{k}^r \am^{i,j}(t,x)\|_{\R^{m\times m}} +  \|(\partial_{k}^r  \bm_{k,n}^{j}(t,x))_{n\geq 1}\|_{\ell^2}
 &\leq C_{\am,\bm}.
\end{align*}
\item non-divergence, $\sigma\in \N_0$: there exist $C_{\am,\bm}>0$, a.s.\ for all $t\in [0,T]$, $i,j,k\in \{1,\dots,d\}$, $k\in \{1,\dots,m\}$, $r\in \{0, \ldots, \sigma\}$,
\begin{align*}
 \|\partial_{k}^r \am^{i,j}(t,x)\|_{\R^{m\times m}} +  \|(\partial_{k}^r  \bm_{k,n}^{j}(t,x))_{n\geq 1}\|_{\ell^2}
 &\leq C_{\am,\bm}.
\end{align*}
\end{itemize}
\end{remark}

\begin{remark}\label{rem:lessregbc}
By Proposition \ref{prop:multiplication_negative}\eqref{it:multiplication_negative3} and Corollary \ref{cor:multiplication_Sobolev} one can check that Steps 2 and 5 of the proof of Lemma \ref{l:estimates_small_interval} still hold under the following regularity assumptions on $b_{k,n}^j$, which are weaker than Assumption \ref{ass:parabolic_problems}\eqref{it:parabolic_problems2}. The case $\sigma=-1$ was already discussed in Remark \ref{r:less_regularity_integer_delta_case} above.
\begin{itemize}
\item Case $\sigma<-1$: there exist $\xi\in [2,\infty)$, $\eta>-\sigma-1$ and $C_b>0$ such that $\eta-\frac{d}{\xi}>0$ and a.s.\ for all $t\in [0,T]$, $j\in \{1,\dots,d\}$ and $k\in \{1,\dots,m\}$,
$$
\|a^{i,j}\|_{H^{\eta,\xi}(\Tor^d;\R^{m\times m})}+
\|(b_{k,n}^j (t,\cdot))_{n\geq 1}\|_{H^{\eta,\xi}(\Tor^d;\ell^2)}\leq C_b.
$$
\item Case $\sigma>-1$: there exist $\xi\in (1,q]$, $\eta\geq 1+\sigma$ and $C_b>0$ such that $\eta-\frac{d}{\xi}\geq 1+\sigma-\frac{d}{q}$, and a.s.\ for all $t\in [0,T]$, $i,j\in \{1,\dots,d\}$, and $k\in \{1,\dots,m\}$,
$$
\|a^{i,j}\|_{H^{\eta,\xi}(\Tor^d;\R^{m\times m})}+
\|(b_{k,n}^j (t,\cdot))_{n\geq 1}\|_{H^{\eta,\xi}(\Tor^d;\ell^2)}\leq C_b.
$$
\end{itemize}
Note that in both cases, $H^{\eta,\xi}(\Tor^d;\ell^2)\embed C^{\varepsilon}(\Tor^d;\ell^2)$ by Sobolev embeddings for some $\varepsilon>0$. Moreover, if $\sigma>-1$, then the threshold case $q=\xi$ and $\eta=1+\sigma$ is allowed.  Note that in the proof of Lemma \ref{l:estimates_small_interval} in the above setting an extension of Lemma \ref{l:extension_operators} to Bessel potential spaces is needed. The latter can be obtained by interpolation from the same type of lemma for Sobolev spaces $W^{1,q}$ with $q\in [1, \infty]$.
\end{remark}

\begin{remark}\label{rem:Rdcase}
We believe that the results of Theorems \ref{t:parabolic_problems} and \ref{t:parabolic_problems_non} also hold if $\T^d$ is replaced by a compact smooth $d$-dimensional manifold without boundary, and the proof should extend almost verbatim.

Theorems \ref{t:parabolic_problems} and \ref{t:parabolic_problems_non} also hold if $\T^d$ is replaced by $\R^d$ if the coefficients become constant for $|x|\to \infty$. More precisely we require that there exist progressive measurable $\wh{a}^{i,j}:[s,T]\times\Omega\to \R^{m\times m}$ and  $\wh{b}^{j}_{k,n}:[s,T]\times\Omega\to \R$ such that
\begin{align*}
\lim_{|x|\to \infty} \esssup_{\omega\in \Omega} \sup_{t\in [s,T]} \big(|a^{i,j}_{k,h} - \wh{a}^{i,j}_{k,h}| +  \|(b^{j}_{k,n} - \wh{b}^{j}_{k,n})_{n\geq 1}\|_{\ell^2} \big) = 0.
\end{align*}
The latter is needed in Step 3 of Lemma \ref{l:estimates_small_interval} with $\Tor^d$ replaced by $\R^d$.

Conditions near infinity can be avoided if $p=q$ see \cite[Theorem 5.1]{Kry} for the case $m=1$, and for an alternative approach \cite[Theorem 5.4]{VP18} which additionally holds in the case $m\geq 1$.
\end{remark}

\subsection{Proof of Lemma \ref{l:estimates_small_interval}}

Before we prove Lemma \ref{l:estimates_small_interval} we first state a simple extension result.
\begin{lemma}
\label{l:extension_operators}
Let $\alpha\in (0,N]$, $\Dom\in \{\R^d,\Tor^d\}$ and let $X$ be a Banach space. Then, for any $y\in \Dom$ and any $r\in (0,\frac{1}{8})$ there exists an extension operator
\[
\e_{y,r}^{\Dom}:C^{\alpha}(\B_{\Dom}(y,r);X)\to C^{\alpha}(\Dom;X)
\]
which satisfies the following properties:
\begin{enumerate}[{\rm(1)}]
\item\label{it:extension_constant} $\e_{y,r}^{\Dom} f|_{\B_{\Dom}(y,r)}=f$, $\e_{y,r}^{\Dom} c\equiv c$ for any $f\in C^{\alpha}(\B_{\Dom}(y,r);X)$ and $c\in X$;
\item\label{it:extension_bounds_C_alpha} $\|\e_{y,r}^{\Dom}\|_{\calL(C^{\alpha}(\B_{\Dom}(y,r);X),C^{\alpha}(\Dom;X))}\leq C_r$ for some $C_r$ independent of $y$;
\item\label{it:extension_bounds_L_infty} $\|\e_{y,r}^{\Dom}\|_{\calL(C(\overline{\B_{\Dom}(y,r)};X),C(\overline{\Dom};X))}\leq C$ for some $C$ independent of $y,r$.
\end{enumerate}
\end{lemma}

\begin{proof}
\textit{Step 1: The case $\Dom=\R^d$}. Note that by localization and the well-known extension operator \cite[Example 1.9 and discussion below it]{InterpolationLunardi}, one can check that there exists a bounded linear operator $\e:C^{\alpha}(\B_{\R^d}(1));E\to C^{\alpha}(\R^d;E)$ such that $\|\e f\|_{L^{\infty}(\R^d;E)}\leq C\|f\|_{L^{\infty}(\B(1);E)}$, $\e c=c$ for any $c\in\R$, and $\e f|_{\B_{\R^d}(1)}=f$ for any $f\in C^{\alpha}(\B_{\R^d}(1))$. Let $\varphi\in C^{\infty}_c(\R^d)$ be radially symmetric and such that $0\leq \varphi\leq 1$, $\varphi|_{\B_{\R^d}(1)}=1$ and $\varphi|_{\R^d\setminus \B_{\R^d}(2)}=0$. For any $f\in C(\overline{\B_{\R^d}}(1);X)$, set
$$(\e_{0,1}^{\R^d}f)(x):=\varphi(x)\e f (x)+(1-\varphi(x))f(0),\qquad x\in \R^d.$$
Then $\e_{0,1}^{\R^d}f=f(0)$ on $\R^d\setminus \B_{\R^d}(2)$ is constant. Setting
$$
\e_{y,r}^{\R^d} f(x)=\e_{0,1}^{\R^d}[f(y+r\cdot)]\Big(\frac{x-y}{r}\Big),\qquad x\in \R^d,
$$
one can check that $\e_{y,r}^{\R^d}$ has the desired properties.

\textit{Step 2: The case $\Dom=\Tor^d$.} We identity the torus with $[-1/2,1/2]^d$. Fix $y\in \Tor^d$ and set
\begin{equation*}
\e_{y,r}^{\Tor^d}(f)(x):= \e_{0,r}^{\R^d}(f(\cdot+y))((x-y) \, \mathrm{mod} 1), \ \ x\in \Tor^d,
\end{equation*}
where we set $(z \, \mathrm{mod} 1)_j = z_j \, \mathrm{mod} 1\in (-1/2,1/2]$ for $z\in \R^d$.
It remains to check that $\e^{\Tor^d}_{y,r}$ satisfies the required properties. Note that $\text{dist}_{\T^d}(x,y) = |(x-y)\, \mathrm{mod} 1|$, and the function $x\mapsto (x-y)\mathrm{mod} 1$ is smooth on the torus. Therefore, $\e_{y,r}^{\Tor^d}$ inherits the required properties from $\e^{\R^d}_{y,r}$.
For instance, if $\text{dist}_{\T^d}(x,y)>2r$, then $\e_{y,r}^{\Tor^d}(f)(x) = f(y)$.
\end{proof}

Next we turn to the proof of Lemma \ref{l:estimates_small_interval}.
\begin{proof}[Proof of Lemma \ref{l:estimates_small_interval}]
We will only prove the divergence form case, since the non-divergence form case is completely analogues. For notational convenience we set $s=0$. Let $\tT>0$ be arbitrary but fixed. It will be chosen below in Step 5.

Fix $t\in [0,T]$, $\ell\in \{0,\a\}$, a stopping time $\tau:\O\to [t,T\wedge (t+\tT)]$, $f\in L^p_{\Progress}((t,\tau)\times\Omega,w_{\ell}^t;X_0)$ and $g\in L^p_{\Progress}((t,\tau)\times\Omega,w_{\ell}^t;\g(\ell^2,X_{1/2}))$ and we set
\begin{equation}
\label{eq:N_f_g_stokes}
N_{f,g}(t,\tau):=\|f\|_{L^p((t,\tau)\times\Omega,w_{\ell}^t;X_0)}+\|g\|_{L^p((t,\tau)\times\Omega,w_{\ell}^t;\g(H,X_{1/2}))}.
\end{equation}
Moreover, for any $y\in\Tor^d$, $r\in (0,\frac{1}{8})$, set $\B(y,r):=\B_{\Tor^d}(y,r)$ and for $v\in H^{2-\reg,q}$
\begin{equation*}
\begin{aligned}
\A_{y}(t)v &:=-\sum_{i,j=1}^d \partial_i (\am^{i,j}(t,y)\partial_j v),
\ &  (\b_{n,y}(t)v)&:=\Big(\sum_{j=1}^d \bm^{j}_{k,n}(t,y)\partial_jv_{k}\Big)_{k=1}^m,\\
\A_{y,r}^{\e}(t)v&:=-\sum_{i,j=1}^d \partial_i (\e_{y,r}^{\Tor^d}(\am^{i,j}(t,\cdot))\partial_j v,
\ &  \b_{n,y,r}^{\e}(t)v&:=\Big(\sum_{j=1}^d(\e_{y,r}^{\Tor^d}(\bm^{j}_{k,n}(t,\cdot)) \partial_j) v_k\Big)_{k=1}^m,
\end{aligned}
\end{equation*}
where $\e_{y,r}^{\Tor^d}$ is the extension operator of Lemma \ref{l:extension_operators}. Comparing the previous definitions with \eqref{eq:def_AB_parabolic}, one sees that $\A_{y},\b_{y}$ are the operators with ``frozen coefficient at $y\in \Tor^d$" and $\A_{y,r}^{\e}$, $\b_{y,r}^{\e}$ are the operators whose coefficients are the extensions of $a^{i,j}|_{\B(y,r)} $, $\btwod^j_n|_{\B(y,r)}$.

To abbreviate the dependencies let
$\Set :=\{q,p,\a,d,\reg,\vartheta,\alpha,C_{a,\btwod}\}$
and we write $C(\Set)$ instead of $C$ below.

\textit{Step 1: There exists $C_1(\Set)>0$ such that for each $t\in [0,T)$, and $y\in \Tor^d$, one has $(\A_{y},\b_{y})\in \mathcal{SMR}_{p,\ell}^\bullet(t,T)$ and}
\begin{equation*}
\max\{C^{\deter,0,p,\ell}_{(\A_{y},\b_{y})}(t,T),C^{\stoc,0,p,\ell}_{(\A_{y},\b_{y})}(t,T)\}\leq C_1.
\end{equation*}
Since the coefficients of $\A_{y},\b_{y}$ are $x$-independent, this follows from the periodic variant of \cite[Theorem 5.3 and Remark 4.6]{VP18}. Moreover, inspecting the proofs of these results one can check that the constants can be taken independent of $t$.

\textit{Step 2: There exists $\eta(\Set)>0$ for which the following holds: }

\textit{If $y\in \Tor^d$ and $r\in (0,\frac{1}{8})$ satisfy}
\begin{equation*}
\|a^{i,j}_{k,h}(t,\cdot)-a^{i,j}_{k,h}(t,y)\|_{L^{\infty}(\B(y,r))}
+\|(b^j_{k,n}(t,\cdot)-b^j_{k,n} (t,y))_{n\geq 1}\|_{L^{\infty}(\B(y,r);\ell^2)}
\leq \eta
\end{equation*}
\textit{a.s.\ for all $t\in [0,T]$, $i,j\in \{1,\dots,d\}$, then $(\A_{y,r}^{\e},\b_{y,r}^{\e})\in\mathcal{SMR}_{p,\ell}^\bullet(t,T)$ and}
\begin{equation*}
\max\{C^{\deter,0,p,\ell}_{(\A_{y,r}^{\e},\b_{y,r}^{\e})}(t,T),C^{\stoc,0,p,\ell}_{(\A_{y,r}^{\e},\b_{y,r}^{\e})}(t,T)\}\leq C_2(\Set).
\end{equation*}
The idea is to apply Theorem \ref{t:pertubation}. To this end, we write
\begin{equation*}
\A_{y,r}^{\e}= \A_{y} +(\A_{y,r}^{\e}-\A_{y}),\qquad
\b_{y,r}^{\e}= \b_{y} +(\b_{y,r}^{\e}-\b_{y}).
\end{equation*}
Let $\varepsilon>0$ be as in Proposition \ref{prop:multiplication_negative}. Then for each $u\in H^{2+\sigma,q}$,
\begin{align*}
\|(\A_{y,r}^{\e}-\A_{y}) u\|_{H^{\sigma,q}}\leq \sum_{i,j=1}^d \sum_{k,h=1}^m \Big\|(a^{i,j}_{k,h}(t,y)-\e_{y,r}^{\Tor^d}(a^{i,j}_{k,h}(t,\cdot)))\partial_j u_h\Big\|_{H^{1+\sigma,q}}.
\end{align*}
We estimate each of the latter terms separately:
\begin{equation*}
\begin{aligned}
\Big\|&(a^{i,j}(t,y) -\e_{y,r}^{\Tor^d}(a^{i,j}(t,\cdot)))\partial_j u\Big\|_{H^{1+\sigma,q}}
\\ &
\stackrel{(i)}{=}
\Big\|\e_{y,r}^{\Tor^d}\big(a^{i,j}(t,y)-a^{i,j}(t,\cdot)\big)\partial_j u\Big\|_{H^{1+\sigma,q}}\\
&
\stackrel{(ii)}{\lesssim}
\Big\|\e_{y,r}^{\Tor^d}\big(a^{i,j}(t,y)-a^{i,j}(t,\cdot)\big)\Big\|_{L^{\infty}}
\|\partial_j u\|_{H^{1+\sigma,q}}\\
& \qquad
+\Big\|\e_{y,r}^{\Tor^d}\big(a^{i,j}(t,y)-a^{i,j}(t,\cdot)\big)\Big\|_{C^{\alpha}}
\|\partial_j u\|_{H^{1+\sigma-\varepsilon,q}}\\
&
\stackrel{(iii)}{\leq}
\eta \|u\|_{H^{2+\sigma,q}}+C_r C_{a,b}\|u\|_{H^{2+\sigma-\varepsilon, q}},
\end{aligned}
\end{equation*}
where in $(i)$, $(iii)$ we used Lemma \ref{l:extension_operators} and in $(ii)$ we used Proposition \ref{prop:multiplication_negative}\eqref{it:multiplication_negative2} and \eqref{it:multiplication_negative4}.
It remains to observe that by interpolation inequalities and Young's inequality
\begin{align*}
\|u\|_{H^{2+\sigma-\varepsilon, q}}\lesssim \|u\|_{H^{2+\sigma, q}}^{1-\frac{\varepsilon}{2}} \|u\|_{H^{\sigma, q}}^{\frac{\varepsilon}{2}}
\leq \eta \|u\|_{H^{2+\sigma, q}} +  C_{\eta,\varepsilon} \|u\|_{H^{\sigma, q}}.
\end{align*}
Similarly, by \cite[Proposition 9.3.1]{Analysis2}
\begin{align*}
\|(\b_{y,r}^{\e}-\b_{y}) u\|_{\gamma(\ell^2,H^{1+\sigma,q})} \lesssim \sum_{j=1}^d \sum_{k=1}^m \Big\|\Big((b^{j}_{k,n}(t,y)-\e_{y,r}^{\Tor^d}(b^{j}_{k,n}(t,\cdot)))\partial_j u_k\Big)_{n\geq 1}\Big\|_{H^{1+\sigma,q}(\ell^2)},
\end{align*}
where $H^{1+\sigma,q} (\ell^2)= H^{1+\sigma,q}(\T^d;\ell^2)$. Each of the terms can be estimated as before, by using Proposition \ref{prop:multiplication_negative}\eqref{it:multiplication_negative2} and \eqref{it:multiplication_negative4} again.

Now the claim of Step 2 follows from Step 1, Theorem \ref{t:pertubation} and the arbitrariness of $\ell\in \{0,\kappa\}$.

\textit{Step 3: Let $\eta$ be as in Step 2. There exist an integer $\Lambda\geq 1$, $(y_\lambda)_{\lambda=1}^\Lambda\subseteq \Tor^d$, $(r_\lambda)_{\lambda=1}^\Lambda\subseteq (0,\frac{1}{8})$, depending only on the quantities in $\Set$, such that $\Tor^d \subseteq \cup_{\lambda=1}^\Lambda \B_\lambda$, where $\B_\lambda:=\B_{\Tor^d}(y_\lambda,r_\lambda)$, and a.s.\ for all $t\in[0,T]$, $i,j\in \{1,\dots,d\}$,}
$$
\|a_{k,h}^{i,j}(t,y_\lambda)-a_{k,h}^{i,j}(t,\cdot)\|_{L^{\infty}(\B_\lambda)}
+
\|(b_{k,n}^j(t,y_\lambda)-b_{k,n}^{j}(t,\cdot))_{n\geq 1}\|_{L^{\infty}(\B_\lambda;\ell^2)}\leq \eta.
$$
\textit{In particular, for all $\lambda\in \{1,\dots,\Lambda\}$ and all $t\in [0,T)$}
\begin{equation*}
(\A_\lambda^{\e},\b_\lambda^{\e}):=(\A^{\e}_{x_\lambda,r_\lambda},\b^{\e}_{x_\lambda,r_\lambda})\in \mathcal{SMR}_{p,\ell}^{\bullet}(t,T),
\end{equation*}
and $\max\{C^{\deter,0,p,\ell}_{(\A_\lambda^{\e},\b_\lambda^{\e})}(t,T),C^{\stoc,0,p,\ell}_{(\A_\lambda^{\e},\b_\lambda^{\e})}(t,T)\}\leq C_3(\Set).
$

The last claim follows from the first one and Step 2. To prove the first claim, fix $y\in \Tor^d$. Note that
\begin{align*}
\|a_{k,h}^{i,j}(t,y)-& a_{k,h}^{i,j}(t,\cdot)\|_{L^{\infty}(\B(y,r))}
+
\|(b_{k,n}^j(t,y)-b_{k,n}^{j}(t,\cdot))_{n\geq 1}\|_{L^{\infty}(\B(y,r);\ell^2)}\\
&
\leq \big([a^{i,j}_{k,h}(t,\cdot)]_{C^{\alpha}(\B(y,r))}+[(b_{k,n}^j(t,\cdot))_{n\geq 1}]_{C^{\alpha}(\B(y,r);\ell^2)}\big) r^{\alpha}\\
&\leq C_{\sigma,q}C_{a,b} \,r^{\alpha}\leq \eta,
\end{align*}
where the last inequality follows by choosing $r:=\min \big\{\big(\frac{\eta}{C_{\sigma,q}C_{a,b }}\big)^{1/\alpha},\frac{1}{8}\big\}$, which only depends on $\Set$. Since $\Tor^d$ can be covered by finitely many balls of the form $B(y,r)$ with $y\in \T^d$, the claim of Step 3 follows.

\emph{Step 4: A representation formula for $u$}. Let $(\phi_\lambda)_{\lambda=1}^{\Lambda}$ be a partition of unity subordinated to the covering $(B_\lambda)_{\lambda=1}^\Lambda$. Multiplying \eqref{eq:parabolic_problem} by $\phi_\lambda$, one obtains
\begin{equation}
\label{eq:parabolic_localized}
\begin{cases}
\dd u_\lambda +\A\, u_\lambda \, \dd t=([\A,\phi_\lambda] u +f_\lambda) \, \dd t \\
\qquad \qquad \qquad\quad + \sum_{n\geq 1} (\b_n u_\lambda+[\phi_\lambda,\b_n]u+g_{n,\lambda}) \, \dd w^n_t,&\text{ on }\Tor^d,\\
u_\lambda(t)=0, &\text{ on }\Tor^d,
\end{cases}
\end{equation}
where $u_\lambda:=\phi_\lambda u$, $f_\lambda:=f \phi_\lambda$, $g_{n,\lambda}:=\phi_\lambda g_n$ and $[\cdot,\cdot]$ denotes the commutator. Since $\supp(u_\lambda)\subseteq \B_\lambda$, one has $\A u_\lambda=\A_{\lambda}^{\e} u_\lambda$, $\b_{n}u_\lambda=\b_{n,\lambda}^{\e} u_\lambda$. By \eqref{eq:parabolic_localized}, Proposition \ref{prop:causality}\eqref{it:causality}, and $(\A_\lambda^{\e},\b_\lambda^{\e})\in \mathcal{SMR}_{p,\ell}^{\bullet}(t,T)$ (see Step 3), one gets
$$
u_\lambda=\Sol_{\lambda}(0,[\A,\phi_\lambda] u, [\phi_\lambda,\b_n]u)+ \Sol_\lambda (0,f_\lambda,(g_{n,\lambda})_{n\geq 1}),\quad \text{ a.e.\ on } [0,\tau]\times \Omega,
$$
where $\Sol_{\lambda}:=\Sol_{0,(\A_\lambda^{\e},\b_\lambda^{\e})}$ is the solution operator associated to the couple $(\A_\lambda^{\e},\b_\lambda^{\e})$, see \eqref{eq:solution_operator_definition}. Therefore,
a.e.\ on $[t,\tau]\times \Omega$,
\begin{equation}
\label{eq:localized_u_parabolic}
\begin{aligned}
u &=\sum_{\lambda=1}^{\Lambda} \Sol_{\lambda}(0,[\A,\phi_\lambda] u, ([\phi_\lambda,\b_n]u)_{n\geq 1})+ \sum_{\lambda=1}^{\Lambda} \Sol_\lambda (0,f_\lambda,(g_{n,\lambda})_{n\geq 1}).
\end{aligned}
\end{equation}

\emph{Step 5: Conclusion}. A straightforward calculation shows that $\|[\phi_\lambda,\A]u\|_{X_0}$ can be estimated by a linear combination of terms of the form $\|a^{ij}_{k,h} (\partial_j \phi_\lambda) u_h\|_{H^{1+\sigma,q}}$ and $\|(\partial_i \phi_\lambda) a^{ij}_{k,h} \partial_j u_h\|_{H^{\sigma,q}}$.
Since $a^{ij}_{k,h}\in C^{\alpha}$ with $\alpha>|1+\sigma|$, Proposition \ref{prop:multiplication_negative}\eqref{it:multiplication_negative2} and \eqref{it:multiplication_negative4}
imply
\begin{align*}
\|a^{ij}_{k,h} (\partial_j \phi_\lambda) u_h\|_{H^{1+\sigma,q}}\lesssim \|u_h\|_{H^{1+\sigma,q}} \lesssim \|u\|_{X_{1/2}}
\end{align*}
For the second term in case $\sigma\leq -1$ choose $\varepsilon\in (0,1]$ such that $\alpha>-\sigma-1+\varepsilon$. Then
\begin{align*}
\|(\partial_i \phi_\lambda) a^{ij}_{k,h} \partial_j u_h\|_{H^{\sigma,q}}\lesssim
\|(\partial_i \phi_\lambda) a^{ij}_{k,h} \partial_j u_h\|_{H^{\sigma+1-\varepsilon,q}} \lesssim \|\partial_j u_h\|_{H^{\sigma+1-\varepsilon,q}} \lesssim \|u\|_{X_{1-\varepsilon/2}}.
\end{align*}
The same estimate holds in case $\sigma>-1$, using $\varepsilon\in (0,\min\{1+\sigma,1\})$.
Similarly since $b^{j}_{k}\in C^{\alpha}(\T^d;\ell^2)$,
\[\|([\phi_\lambda,\b_n] u)_{n\geq 1}\|_{\gamma(\ell^2,X_{1/2})}\lesssim \sum_{j=1}^d\sum_{k=1}^m \|(b_{k,n} (\partial_j\phi_\lambda) u_k)_{n\geq 1}\|_{H^{1+\sigma,q}(\ell^2)}\lesssim \|u\|_{X_{1/2}}.\]

From \eqref{eq:localized_u_parabolic} and Step 3 it follows that
\begin{align*}
\|u\|_{L^p((t,\tau)\times\Omega,w_{\ell}^t;X_1)} \lesssim_{\Set} T_1+T_2 + N_{f,g}(t,\tau),
\end{align*}
where $N_{f,g}(t,\tau)$ is as in \eqref{eq:N_f_g_stokes}, and
\begin{align*}
T_1 & :=\max_{\lambda\in \{1,\dots,\Lambda\}}
 \|[\A,\phi_\lambda]u\|_{L^p((t,\tau)\times\Omega,w_{\ell}^t;X_0)} \lesssim_{\Set} \|u\|_{L^p((t,\tau)\times\Omega,w_{\ell}^t;X_{1-\varepsilon/2})},
\\  T_2 &:=\max_{\lambda\in \{1,\dots,\Lambda\}}
 \|([\phi_\lambda,\b_n] u)_{n\geq 1}\|_{L^p((t,\tau)\times\Omega,w_{\ell}^t;\gamma(\ell^2,X_{1/2}))} \lesssim_{\Set} \|u\|_{L^p((t,\tau)\times\Omega,w_{\ell}^t;X_{1/2})},
\end{align*}
for some $\varepsilon(q,\sigma,d)\in (0,1]$. Thus by standard interpolation inequalities (see \cite[(C.1)]{Analysis1}), Holder's and Young's inequality there exists a constant $C_5(\Set)>0$ such that
\begin{equation}
\label{eq:estimate_with_lower_order_terms}
\|u\|_{L^p((t,\tau)\times\Omega,w_{\ell}^t;X_1)}\leq \frac{1}{4} \|u\|_{L^p((t,\tau)\times\Omega,w_{\ell}^t;X_1)}+C_5 \Big(\|u\|_{L^p((t,\tau)\times\Omega,w_{\ell}^t;X_0)}+  N_{f,g}(t,\tau)\Big),
\end{equation}
To obtain the desired estimate we will use that $\tau\leq (t+\tT)\wedge T$ and choose $\tT$ small enough. Indeed, due to the first statement in \cite[Lemma 3.13]{AV19_QSEE_1}, there exists constants $(c_{s}(\Set))_{s>0}$ independent of $t$ such that $\lim_{s\downarrow 0}c_s=0$ and
\begin{equation}
\label{eq:estimate_to_control_lower_order_terms}
 \|u\|_{L^p((t,\tau)\times\Omega,w_{\ell}^t;X_0)}\leq c_{\tT} (\|u\|_{L^p((t,\tau)\times\Omega,w_{\ell}^t;X_1)}+ N_{f,g}(t,\tau)).
\end{equation}
By choosing $\tT$ so small that $c_{\tT}\leq \frac{1}{4C_5}$ combining \eqref{eq:estimate_with_lower_order_terms} and \eqref{eq:estimate_to_control_lower_order_terms} it follows that $\|u\|_{L^p((t,\tau)\times\Omega,w_{\ell}^t;X_1)}\leq (2C_5+\frac12) N_{f,g}(t,\tau)$. This completes the proof.
\end{proof}


\begin{acks}[Acknowledgments]
The authors thank the anonymous referees and Max Sauerbrey for careful reading and helpful suggestions.
\end{acks}

\begin{funding}
The first author has been partially supported by the Nachwuchsring -- Network for the promotion of young scientists -- at TU Kaiserslautern.
The second author is supported by the VIDI subsidy 639.032.427 of the Netherlands Organisation for Scientific Research (NWO)
\end{funding}


\bibliographystyle{imsart-number}
\bibliography{literature}

\end{document}